\documentclass[10pt,reqno]{amsart}  
\usepackage{amsfonts, amssymb,stmaryrd}
\usepackage{showlabels}
\usepackage{latexsym}
\usepackage{graphicx} 
\input xy
\xyoption{all}
  
\newtheorem{theorem}{Theorem}[section]
\newtheorem{corollary}[theorem]{Corollary}
\newtheorem{lemma}[theorem]{Lemma}
\newtheorem{proposition}[theorem]{Proposition}

\theoremstyle{definition}
\newtheorem{definition}[theorem]{Definition}
\newtheorem{example}[theorem]{Example}
\theoremstyle{remark}
\newtheorem{remark}[theorem]{Remark}

\newtheorem{note}[theorem]{Note}
\numberwithin{equation}{section}

\newcommand{\GL}{\mathrm{GL}}

\newcommand{\Gr}{\mathbf{Gr}}

\newcommand{\Sub}{\mathrm{Sub}}
\newcommand{\Env}{\mathrm{Env}}

\newcommand{\Rec}{\mathrm{Rec}}
\newcommand{\SB}{\mathrm{SB}}
\newcommand{\Base}{\mathrm{Base}}

\newcommand{\Me}{\mathrm{St}}

\newcommand{\chiNF}{\chi}

\newcommand{\Tail}{\mathrm{Tail}}

\newcommand{\Fix}{\mathrm{Fix}}
\newcommand{\CAT}{\mathrm{CAT}}

\newcommand{\Lie}{\mathrm{Lie}}

\newcommand{\PSL}{\mathrm{PSL}}
\newcommand{\PGL}{\mathrm{PGL}}
\newcommand{\Ad}{\mathrm{Ad}}
\newcommand{\Aut}{\mathrm{Aut}}

\newcommand{\IRS}{\mathrm{IRS}}

\newcommand{\Out}{\mathrm{Out}}

\newcommand{\Hd}{\mathrm{Hd}}

\newcommand{\Rgf}{\overline{H}_g}
\newcommand{\Rgfi}{\overline{H}_{g^{-1}}}

\newcommand{\arrow}{\rightarrow}
\newcommand{\compos}{\circ}
\newcommand{\trivgp}{\langle e \rangle}

\newcommand{\R}{\mathbf R}

\newcommand{\Q}{\mathbf Q}
\newcommand{\N}{\mathbf N}
\newcommand{\Z}{\mathbf Z}
\newcommand{\F}{\mathbf F}
\newcommand{\I}{\mathbf I}
\newcommand{\Fc}{\mathcal{F}}

\newcommand{\Bc}{\mathcal{B}}
\newcommand{\Ac}{\mathcal{A}}
\newcommand{\Rc}{\mathcal{R}}

\newcommand{\Eg}{\operatorname{\mathcal{E}}}
\newcommand{\Fg}{\operatorname{\mathcal{F}}}
\newcommand{\Fgg}{\operatorname{\mathcal{G}}}
\newcommand{\Hg}{\operatorname{\mathcal{H}}}

\renewcommand{\Mc}{\operatorname{Prob}}
\newcommand{\Fin}{\operatorname{Fin}}
\newcommand{\PP}{\mathbb P}
\newcommand{\G}{\mathbf G}

\renewcommand{\H}{\mathbf H}

\newcommand{\leftIRS}{\leftslice}

\newcommand{\Rad}{\operatorname{Rad}}
\newcommand{\chr}{\operatorname{char}}

\newcommand{\gc}{\mathfrak{g}}
\newcommand{\hc}{\mathfrak{h}}

\newcommand{\norm}[1]{\left\Vert#1\right\Vert}

\usepackage[pdftex,bookmarks,colorlinks]{hyperref}
\begin{document}
\title{Invariant random subgroups of linear groups}%
\author{Yair Glasner}

%



\maketitle 
\begin{center} With an appendix by Tsachik Gelander and Yair Glasner \end{center}
\begin{abstract}
Let $\Gamma < \GL_n(F)$ be a countable non-amenable linear group with a simple, center free Zariski closure. Let $\Sub(\Gamma)$ denote the space of all subgroups of $\Gamma$ with the, compact, metric, Chabauty topology. An {\it{invariant random subgroup}} (IRS) of $\Gamma$ is a conjugation invariant Borel probability measure on $\Sub(\Gamma)$. An $\IRS$ is called {\it{nontrivial}} if it does not have an atom in the trivial group, i.e. if it is nontrivial almost surely.  We denote by $\IRS^{0}(\Gamma)$ the collection of all nontrivial $\IRS$ on $\Gamma$.
\begin{theorem} 
With the above notation, there exists a free subgroup $F < \Gamma$ and a non-discrete group topology on $\Gamma$ such that for every $\mu \in \IRS^{0}(\Gamma)$ the following properties hold:
\begin{itemize}
\item $\mu$-almost every subgroup of $\Gamma$ is open.
\item $F \cdot \Delta = \Gamma$ for $\mu$-almost every $\Delta \in \Sub(\Gamma)$.
\item $F \cap \Delta$ is infinitely generated, for every open subgroup.  In particular this holds for $\mu$-almost every $\Delta \in \Sub(\Gamma)$.
\item The map
\begin{eqnarray*}
\Phi: \left(\Sub(\Gamma),\mu \right) & \arrow & \left(\Sub(F),\Phi_* \mu \right) \\
\Delta & \mapsto & \Delta \cap F
\end{eqnarray*}
is an $F$-invariant isomorphism of probability spaces.  
\end{itemize}
\end{theorem}
A more technical version of this theorem is valid for general countable linear groups. We say that an action of $\Gamma$ on a probability space, by measure preserving transformations, is {\it{almost surely non free}} (ASNF) if almost all point stabilizers are non-trivial. 
\begin{corollary} \label{cor:prod}
Let $\Gamma$ be as in the Theorem above. Then the product of finitely many ASNF $\Gamma$-spaces, with the diagonal $\Gamma$ action, is ASNF. 
\end{corollary}
\begin{corollary} 
Let $\Gamma < \GL_n(F)$ be a countable linear group, $A \lhd \Gamma$ the maximal normal amenable subgroup of $\Gamma$ - its {\it{amenable radical}}. If $\mu \in \IRS(\Gamma)$ is supported on amenable subgroups of $\Gamma$ then in fact it is supported on $\Sub(A)$. In particular if $A(\Gamma) = \trivgp$ then $\Delta = \trivgp, \mu$ almost surely. \end{corollary}
\end{abstract}
\newpage
\tableofcontents 
\section{Introduction}
Before discussing our main theorem and some corollaries, we need to develop the language of invariant random subgroups. 
\subsection{Invariant random subgroups}
 Let $\Gamma$ be a discrete countable group and $\Sub(\Gamma)$ the collection of all subgroups of $\Gamma$ endowed with the compact Chabauty topology. A basis for this topology is given by sets of the form 
$$U_{M}(C) = \left \{ D \in \Sub(\Gamma) \ \left | \ D \cap M = C \cap M \right. \right \}.$$
Where $C \in \Sub(\Gamma)$ and $M \subset \Gamma$ is a finite subset. In other words, $\Sub(\Gamma) \subset 2^{\Gamma}$ inherits its  topology from the Tychonoff topology on $2^{\Gamma}$. The last definition shows that $\Sub(\Gamma)$ is a compact metrizable space. Every subgroup $\Sigma \in \Sub(\Gamma)$ gives rise to two closed subsets, the collection of its subgroups and supergroups
\begin{eqnarray*}
\Sub(\Sigma) & = & \left\{\Delta \in \Sub(\Gamma) \ | \ \Delta < \Sigma \right \} \subset  \Sub(\Gamma) \\
\Env(\Sigma) & = & \left \{\Delta \in \Sub(\Gamma) \ | \ \Sigma < \Delta \right \} \subset \Sub(\Gamma).
\end{eqnarray*}
The latter is referred to as the {\it{envelope}} of $\Sigma$, it is open whenever $\Sigma$ is finitely generated. The collection of open sets
$$\left \{\Env(\Sigma) \ | \ \Sigma \in \Sub(\Gamma) {\text{ f.g.}} \right \} \bigcup \left \{ \Sub(\Gamma) \setminus \Env(\Sigma) \ | \ \Sigma \in \Sub(\Gamma) {\text{ f.g.}} \right \}.$$
forms a basis for the topology of $\Sub(\Gamma)$ since: 
$$U_M(C) = \Env(\langle M \cap C \rangle) \cap \left(\bigcap_{\gamma \in M \setminus C} \Sub(\Gamma) \setminus \Env(\langle \delta \rangle) \right).$$
The group $\Gamma$ acts (continuously from the left) on $\Sub(\Gamma)$ by conjugation $\Gamma \times \Sub(\Gamma) \arrow  \Sub(\Gamma)$ namely $(g,H)  \mapsto gHg^{-1}$. The following definition plays a central role in all of our discussion: 
\begin{definition}
Let $\Gamma$ be a countable group. An {\it{invariant random subgroup}} or an {\it{IRS}} on $\Gamma$ is a Borel probability-measure $\mu \in \Mc(\Sub(\Gamma))$ which is invariant under the $\Gamma$ action. We denote the space of all invariant random subgroups on $\Gamma$ by $\IRS(\Gamma) = \Mc_{\Gamma}(\Sub(\Gamma))$. 

An IRS is called {\it{nontrivial}} if it does not have an atom at the trivial group $\trivgp$. The collection of all nontrivial IRS will be denoted $\IRS^{0}(\Gamma)$. An IRS is ergodic if it is an ergodic measure with respect to the action of $\Gamma$. 

As is customary in probability theory, we will sometimes just say that $\Delta$ is an IRS in 
$\Gamma$ and write $\Delta \leftIRS \Gamma$, when we mean that some $\mu \in \IRS(\Gamma)$ has been implicitly fixed and that $\Delta \in \Sub(\Gamma)$ is a $\mu$-random sample. Thus in our terminology above saying that an IRS is nontrivial means that it is nontrivial almost surely. 
\end{definition}

\begin{remark}
The theory of invariant random subgroups can be developed also in the more general setting where $\Gamma$ is a locally compact second countable group. In that case $\Sub(\Gamma)$ is taken to be the space of all {\it{closed}} subgroups of $\Gamma$ and it is again a compact metrizable space. We will not pursue this more general viewpoint here. We refer the interested reader to the paper \cite{7_sam:ann} and the references therein. 
\end{remark}
The term ``IRS'' was introduced in a pair of joint papers with Ab\'{e}rt and Vir\'{a}g \cite{AGV:Kesten_Measureable, AGV:Kesten_IRS}, however the paper of Stuck-Zimmer \cite{SZ} is quite commonly considered as the first paper on this subject. That paper provides a complete classification of IRS in a higher rank simple Lie group $G$, by showing that every ergodic IRS is supported on a single orbit (i.e. conjugacy class), either of a central normal subgroup or of a lattice. A similar classification is given of the ergodic IRS of a lattice $\Gamma < G$. These are supported either on a finite central subgroup or on the conjugacy class of a finite index subgroup. Recent years have seen a surge of activity in this subject, driven by its intrinsic appeal based on the interplay between group theory and ergodic theory, as well as by many applications that were found. We mention in combinatorics and probability \cite{AGV:Kesten_Measureable, AL:unimodular_random_networks,LP:cycle_density,Cannizzo:inv_sch}, representation theory and asymptotic invariants \cite{7_sam:ann,7_sam,BG:Betti,Ram:arith_orbi},  dynamics \cite{bowen:furst_ent,Verskik:tnf}, group theory \cite{AGV:Kesten_IRS,Vershik:characters,bowen:irs_free,BGK:irs_lamp,BGK:char_irs, GM:top_full_gp,KN:subset_currents,TTd:finite_sym} and rigidity \cite{SZ,Bekka:SZ,PT:stab_erg,CP:char_rigidity,CP:stab_erg,HT:rigidity,TD:shift_minimal,Cr:stab_prod}. Topological analogues of IRS were introduced in \cite{GW:URS}. 

The current paper initiates a systematic study of the theory of IRS in countable linear groups. The theory of linear groups is known as a benchmark of sorts in group theory. Contrary to the wild nature of general countable groups, linear groups support a very elaborate structure theory. On the other hand they exhibit a very wide array of phenomena. They are not too specialized, compared for example with abelian, nilpotent or even amenable groups. In this capacity methods and questions that were addressed in the setting of linear groups are often pushed further to other groups of ``geometric nature''. Hyperbolic and relatively hyperbolic groups, acylindrically hyperbolic groups, convergence groups, mapping class groups, $\Out(F_n)$, and even just residually finite groups are some of the examples. Undoubtably many of the methods developed here can be applied in many of these other geometric settings. 

One question which is inaccessible for linear groups, or for any other family of groups that is not extremely special, is the study of all subgroups of a given group $\Gamma$. This is out of reach even for lattices in higher rank simple Lie groups, where we do have a complete understanding of seemingly similar problems such as the classification of all quotient groups or of all finite dimensional representations. An outcome of the Stuck-Zimmer paper \cite{SZ} is that, in the presence of an invariant measure - an IRS - the situation changes dramatically. For a lattice in a simple Lie group almost every subgroup, with respect to any IRS, is either finite central or of finite index. We would like to draw the attention to this formation of quantifiers. It will appear frequently, as we attempt to follow a similar path, proving in the setting of a countable linear group $\Gamma$, statements that hold for almost every subgroup with respect to every IRS. 

\subsection{Examples} \label{sec:eg_IRS}
\begin{example} \label{eg:normal}
{\it{Normal subgroups:}} If $N \lhd \Gamma$ is a normal subgroup then the Dirac measure $\delta_N$ supported on the single point $N \in \Sub(\Gamma)$ is an IRS.
\end{example}
\begin{example} \label{eg:fin}
{\it{Almost normal subgroups:}} If $H < \Gamma$ is such that $[\Gamma: N_{\Gamma}(H)] < \infty$ then we will say that $H$ is an {\it{almost normal subgroup}}. In this case a uniformly chosen random conjugate of $H$ is an IRS. Normal and finite index subgroups are both examples of almost normal subgroups.  
\end{example}
\begin{example} \label{eg:pmp}
{\it{Probability-measure preserving actions:}} If $\Gamma \curvearrowright (X, \Bc, \mu)$ is an action on a probability space by measure preserving  transformations then the stabilizer $\Gamma_x := 
\{ \gamma \in \Gamma \ | \ \gamma x=x \} \leftIRS \Gamma$ of a $\mu$-random point $x \in X$ is an IRS. The probability measure responsible for this IRS is $\Phi_*(\mu)$ where $\Phi: X \arrow \Sub(\Gamma)$ is the stabilizer map $x \mapsto \Gamma_x$. 
Note that such an IRS is nontrivial if and only if it comes from an action that is almost surely non free. Following our convention that and IRS is called nontrivial if it is nontrivial almost surely. 

It was shown in \cite[Proposition 13]{AGV:Kesten_IRS} that every invariant random subgroup of a finitely generated group is obtained in this fashion: 
\begin{lemma} \label{lem:IRS}
For every invariant random subgroup $\nu \in \IRS(\Gamma)$ of a countable group there exists a probability-measure preserving action $\Gamma \curvearrowright (X,\Bc,\mu)$ such that $\nu = \Phi_*(\mu)$ where $\Phi: X \arrow \Sub(\Gamma)$ is the stabilizer map $x \mapsto \Gamma_x$.
\end{lemma}
\noindent This was later generalized to the setting of locally compact second countable groups in \cite[Theorem 2.4]{7_sam}.  
\end{example}

\subsection{Induction, restriction and intersection} \label{sec:ind_res}
Let $\Sigma < \Gamma$ be a subgroup. The natural, $\Sigma$-invariant, restriction map $R^{\Gamma}_{\Sigma}: \Sub(\Gamma)  \arrow  \Sub(\Sigma)$ given by the intersection with $\Sigma$ is defined by the formula $\Delta \mapsto \Delta \cap \Sigma$. This map is clearly continuous and it gives rise to a map on invariant random subgroups $\left(R^{\Gamma}_{\Sigma} \right)_*:\IRS(\Gamma) \arrow \IRS(\Sigma)$. We will denote this map by $\mu \mapsto \mu|_{\Sigma}$. Thus for any $\mu \in \IRS(\Gamma)$ the original restriction map becomes a $\Sigma$ invariant map of probability spaces:
\begin{eqnarray*}
R^{\Gamma}_{\Sigma}: \left(\Sub(\Gamma),\mu \right) & \arrow & \left(\Sub(\Sigma),\mu|_{\Sigma} \right) \\
\Delta & \mapsto & \Delta \cap \Sigma
\end{eqnarray*}

A map in the other direction exists only if $\Sigma$ is of finite index in $\Gamma$. In this case we obtain an induction map
$\IRS(\Sigma) \arrow \IRS(\Gamma)$ given by 
$$\mu \mapsto\mu|^{\Gamma} = \frac{1}{[\Gamma:\Sigma]} \sum_{i = 1}^{[\Gamma: \Sigma]} (\gamma_i )_* \mu,$$
where $\gamma_i$ are coset representatives for $\Sigma$ in $\Gamma$. Since the measure $\mu$ is $\Sigma$ invariant this expression does not depend on the choice of these representatives. In the setting of locally compact groups this construction can be generalized to the case where $\Sigma$ is a lattice in $\Gamma$, or more generally a subgroup of co-finite volume. All one has to do is replace the sum by an integral over $\Gamma/\Sigma$ with respect to the invariant probability measure. 

The {\it{intersection $\mu_1 \cap \mu_2$}} of $\mu_1,\mu_2 \in \IRS(\Gamma)$ is defined as the push forward of the product measure under the (continuous) map
\begin{eqnarray*}
\Sub(\Gamma) \times \Sub(\Gamma) & \arrow & \Sub(\Gamma) \\
(\Delta_1,\Delta_2) & \mapsto & \Delta_1 \cap \Delta_2.
\end{eqnarray*}
This can also be thought of as restricting the IRS $\mu_1$ to a $\mu_2$-random subgroup or vice versa. 

\subsection{The main theorems}
In our main theorem below we adopt the following convention. If $\Sub(\Gamma) = S_0 \sqcup S_1 \sqcup S_2 \sqcup \ldots \sqcup S_L$ is a partition into $\Gamma$-invariant sets and if $\mu \in \IRS(\Gamma)$ we set $a_{\ell} = \mu(S_{\ell})$ and $\mu_{\ell}(A) = \mu(A \cap S_{\ell})/a_{\ell} \in \IRS(\Gamma)$, so that $\mu = \sum_{\ell = 0}^{L} a_{\ell} \mu_{\ell},$
is the standard decomposition of $\mu$ as a convex combination of IRS supported on these parts. 
\begin{theorem} \label{thm:main_irs} (Main theorem, IRS version)
Let $\Gamma < \GL_n(F)$ be a countable linear group with a connected Zariski closure and $A= A(\Gamma)$ its amenable radical. Then there exists a number $L = L(\Gamma) \in \N$, proper free subgroups $\{F_1,F_2, \ldots, F_L\} \subset \Sub(\Gamma)$ and a  partition into $\Gamma$-invariant subsets $\Sub(\Gamma) = S_0 \sqcup S_1 \sqcup \ldots \sqcup S_L$ such that for every $\mu \in \IRS(\Gamma)$ the following properties hold:
\begin{description} 
\item[Amm] \label{itm:Amm} $\Delta < A$ for $\mu_0$-almost every $\Delta \in \Sub(\Gamma)$.
\item[Free] \label{itm:Free} $F_{\ell} \cap \Delta$ is a non-abelian infinitely generated free group, for every $1 \le \ell \le L$ and $\mu_{\ell}$-almost every $\Delta \in \Sub(\Gamma)$. 
\item[Me-Dense] \label{itm:Dense} $F_{\ell} \cdot \Delta = \Gamma$, for every $1 \le \ell \le L$ and $\mu_{\ell}$-almost every $\Delta \in \Sub(\Gamma)$. 
\item[Isom] \label{itm:Isom} The map 
\begin{eqnarray*}
R^{\Gamma}_{F_{\ell}}: (\Sub(\Gamma),\mu_{\ell}) & \rightarrow & (\Sub(F_{\ell}),\mu_{\ell}|_{F_{\ell}}) \\
\Delta & \mapsto & \Delta \cap F
\end{eqnarray*} 
is an ($F_{\ell}$-equivariant) isomorphism of probability spaces, for every $1 \le \ell \le L$ and $\mu_{\ell}$-almost every $\Delta \in \Sub(\Gamma)$. \footnote{Explicitly this means that, after restricting to conullsets in the domain and range this becomes a measurable, measure preserving bijection with measurable inverse.}. 
\end{description}
Finally, it is possible to satisfy the conclusion of the theorem with $L=1$ if and only if $\Gamma/A$ contains no nontrivial commuting almost normal subgroups (as in Example \ref{eg:fin}). 
\end{theorem}

\begin{remark}
The free groups in the theorem are in general not finitely generated. 
\end{remark}
\begin{remark} \label{rem:free_follows}
Property {\bf{Free}} of the theorem follows directly from property {\bf{Isom}} when the measure has no atoms since such a measure will always give measure zero to the countable subset of finitely generated subgroups of $F$. 
\end{remark}
\begin{remark} \label{rem:ell}
Sometimes we will use the notation {\bf{$\ell$-Free}} to denote the property {\bf{Free}} in the theorem above for a specific value of $1 \le \ell \le L$, and similarly for the other properties. 
\end{remark}

\noindent An important corollary of the main theorem is:
\begin{corollary} \label{cor:main_irs}
Let $\Gamma < \GL_n(F)$ be a countable linear group. If $\Delta \leftIRS \Gamma$ is an $\IRS$ in $\Gamma$ that is almost surely amenable, then $\Delta < A(\Gamma)$ almost surely. 
\end{corollary}
This theorem was vastly generalized in a beautiful paper of Bader, Duchesne and L\'{e}cureux \cite{BDL:amenable_irs} who prove the exact same theorem {\it{without}} the linearity assumption.

In the Appendix (\ref{appendix}), jointly with Tsachik Gelander, we give a short proof for this corollary. This short proof avoids many of the technicalities of the main theorem as well as all of the projective dynamical argument.  

The main Theorem \ref{thm:main_irs} assumes its most appealing form when $A(\Gamma) = \trivgp$ and $L(\Gamma)=1$. As mentioned in the abstract these conditions are satisfied whenever $\Gamma$ is nonamenable and has a simple, center free, Zariski closure. Before stating the theorem in this specific case we note the following:
\begin{remark} \label{rem:ast_ammenable} If $\Gamma < \GL_n(F)$ is not finitely generated and $\chr(F)>0$ then it is possible for an amenable group $\Gamma$ to have a simple Zariski closure. An example is the group $\PSL_n(F)$ when $F$ is a locally finite field, such as the algebraic closure of $\F_p$. This is the reason why many Tits-alternative type proofs restrict themselves to finitely generated groups in positive characteristic. In Theorem \ref{thm:main_irs} we were are able to subsume the amenable case into the statement of the general theorem.  \end{remark}
\begin{definition} \label{def:chiNF}
Given a representation $\chi: \Gamma \arrow \GL_n(\Omega)$ we say that $\mu \in \IRS(\Gamma)$ is {\it{$\chi$ non free}} if $\chi(\Delta) \ne \trivgp$ for $\mu$ almost every $\Delta \in \Sub(\Gamma)$. We denote the collection of all such IRS by 
$$\IRS^{\chiNF}(\Gamma) = \{\mu \in \IRS(\Gamma) \ | \ \mu(\Sub(\ker(\chi))) = 0 \}.$$
\end{definition}
\begin{theorem} \label{thm:one_ell} (Main theorem, simple version)
Let $\Gamma$ be a nonamenable countable group and $\chi:\Gamma \arrow \GL_n(F)$ a linear representation whose image has a simple, center free, Zariski closure. Then there is a group topology $\Me_{\chi}$ on $\Gamma$ such that $\mu$ almost every subgroup $\Delta \in \Sub(\Gamma)$ is open for every $\mu \in \IRS^{\chiNF}(\Gamma)$. Moreover there exists a proper free subgroup $F < \Gamma$ such that the following properties hold for every $\mu \in \IRS^{\chiNF}(\Gamma)$:
\begin{description}
\item[{\bf{$\chi$-Me-Dense}}] $F \cdot \Delta = \Gamma, {\text{ for }} \mu {\text{ almost every }} \Delta \in \Sub(\Gamma)$
\item[{\bf{$\chi$-Non-Disc}}] $F \cap \Delta \ne \trivgp$ for every open subgroup $\Delta \in \Sub(\Gamma)$
\item[{\bf{$\chi$-Isom}}] The restriction map 
\begin{eqnarray*}
R^{\Gamma}_{F}: (\Sub(\Gamma),\mu) & \rightarrow & (\Sub(F),\mu|_{F}) \\
\Delta & \mapsto & \Delta \cap F
\end{eqnarray*} 
is an $F$-invariant isomorphism of probability spaces for every $\mu \in \IRS^{\chiNF}(\Gamma)$. 
\end{description}
\end{theorem}
In the more specific setting in which this theorem is set, it is actually stronger than Theorem \ref{thm:main_irs} in two ways. It proves the existence of a topology and it no longer requires the Zariski closure to be connected. 

The group topology defined above is called the {\it{$\chi$-stabilizer topology}} on $\Gamma$. When $\chi$ is injective, we just refer to the {\it{stabilizer topology}} $\Me$ on $\Gamma$. The actual definition of the topology is quite technical and we postpone it to Subsection  \ref{sec:stab_top}. A special case where $\Gamma = F_r$ is a nonabelian free group is particularly interesting. This case is treated separately in Subsection \ref{sec:free} which gives a {\it{conceptual understanding of the topology, while avoiding the technical complications}}. 

We call any group $F < \Gamma$ satisfying the property {\bf{Me-Dense}} appearing in the theorem {\it{measurably dense}}. Note though that this property is, at least a-priori {\it{weaker than actual density in the stabilizer topology}}. Property {\bf{Non-Disc}} on the other hand is equivalent to the assertion that $e$ is an accumulation point of $F \setminus \{e\}$. In particular it shows that the stabilizer topology is non-discrete. In fact $F$ is a non discrete subgroup of $\Gamma$ in the stabilizer topology.

\subsection{Non-free actions on probability spaces}
In view of Example \ref{eg:pmp} we can restate the main theorems in terms of actions on probability spaces without any reference to IRS. In the main theorem below we choose to pass to a finite index subgroup, instead of restricting to groups with a connected Zariski closure. The proofs are straightforward and are left to the reader. 
\begin{theorem}  \label{thm:main} (Main theorem, p.m.p version)
Let $\Gamma < \GL_n(F)$ be a countable linear group, $A=A(\Gamma)$ its amenable radical. Then there exists a finite index subgroup $\Gamma^0<\Gamma$, a number $L \in \N$ and proper free subgroups $F_1, F_2, \ldots, F_L < \Gamma$ such that for every action, by measure preserving transformations, on a standard Borel probability space, $\Gamma \curvearrowright (X,\Bc,\mu)$ there exists a Borel measurable partition into $\Gamma^0$-invariant subsets:
$$X = X_0 \sqcup X_1 \sqcup X_2 \sqcup \ldots \sqcup X_L,$$  
with the following properties:
\begin{itemize}
\item $\Gamma_x < A$ for almost every $x \in X_0$. 
\item $F_{\ell} \cdot x \supset \Gamma^0 \cdot x$ for every $1 \le \ell \le L$ and almost every $x \in X_{\ell}$. 
\item $F_{\ell} \cap \Gamma_x$ is an infinitely generated free group for every $1 \le \ell \le L$ and almost every $x \in X_{\ell}$.
\item For every $1 \le \ell \le L$ and almost every pair of points $(x,y) \in X_{\ell}^2$ we have that 
$$\left(F_{\ell} \cap \Gamma_x = F_{\ell} \cap \Gamma_y\right) \Rightarrow \left(\Gamma_x = \Gamma_y \right) $$
\end{itemize}
Moreover if $\Gamma/A$ contains no nontrivial elementwise commuting normal subgroups then we can take $L=1$ and $\Gamma^0  = \Gamma$ in the above theorem. 
\end{theorem}
Recall from the abstract that an action $\Gamma \curvearrowright (X,\Bc,\mu)$ is almost surely non-free (ASNF) if almost all stabilizers $\Gamma_x$ are non-trivial. 
\begin{theorem} (simple p.m.p version)
Assume that $\Gamma$ is a countable non-amenable linear group with a simple center free Zariski closure. Then there exists a free subgroup $F < \Gamma$ with the property that for every ASNF action of $\Gamma$ the two groups $\Gamma$ and $F$ have almost surely the same orbits. In other words the actions of $\Gamma$ and $F$ are orbit equivalent.
\end{theorem}

\begin{corollary} (Actions with amenable stabilizers are free) \label{cor:main}
Let $\Gamma < \GL_n(F)$ be a countable linear group and $\Gamma \curvearrowright (X,\Bc,\mu)$ a measure preserving action on a standard Borel probability space such that $\Gamma_x$ is almost surely amenable. Then $\Gamma_x < A$ almost surely. In particular if $\Gamma$ has a trivial amenable radical every action with amenable stabilizers is  essentially  free.
\end{corollary}
As in Corollary \ref{cor:main_irs}, here too the linearity assumption on $\Gamma$ is unnecessary due to the beautiful work of Bader-Duchesne-L\'{e}cureux \cite{BDL:amenable_irs}. 

The non trivial intersection of almost every two instances of (possibly equal) IRS's translates to the following:
\begin{theorem} (Non freeness of product actions)
Let $\Gamma$ be a countable, non-amenable linear group with a simple center free Zariski closure. If $\alpha_i: \Gamma \arrow \Aut(X,\Bc,\mu), \ \ 1 \le i \le M$ are ASNF probability measure preserving actions then so is the product action $\Pi_{i=1}^M: \Gamma \arrow (X^M, \Bc^M, \mu^M)$.
\end{theorem}

We conclude the introduction by mentioning two places where the results of this paper have been used.  In a beautiful paper \cite[Corollary 5.11]{TD:shift_minimal} Robin Tucker-Drob shows that for a countable linear group $\Gamma$ the condition $A(\Gamma) = \trivgp$ is equivalent to the fact that any probability measure preserving action of $\Gamma$ that is weakly contained in the Bernoulli shift $\Gamma \curvearrowright [0,1]^{\Gamma}$ is actually weakly equivalent to the Bernoulli. This last property is what Tucker-Drob refers to as the {\it{shift minimality}} of the group $\Gamma$ which has many other equivalent formulations \cite[Proposition 3.2]{TD:shift_minimal}. 

Corollary \ref{cor:main_irs} was also used in \cite[Theorem 5]{AGV:Kesten_IRS} to show that if $\Gamma$ is a finitely generated linear group with $A(\Gamma) = \trivgp$ and $X_n$ is a sequence of Schreier graphs for which the spectral radius of the random walk on $L_2^{0}(X_n)$ converges to the spectral radius of $\Gamma$ then $X_n$ also converge to $\Gamma$ in the sense of Benjamini-Schramm. Both of these results are now valid for {\it{general countable groups}}, due to the Bader-Duchesne-L\'{e}cureux Theorem \cite{BDL:amenable_irs}.

\section{Essential subgroups and recurrence}
\label{sec:essential}
\subsection{Essential and locally essential subgroups}
The definition of locally essential subgroups, and the lemma that follows are, in my opinion, the most important new tools introduced in this paper. These notions are specific to the theory of IRS in countable groups and so far they do not have satisfactory generalizations to the locally compact case. 
\begin{definition}
Let $\mu \in \IRS(\Gamma)$, a subgroup $\Sigma < \Gamma$ is called $\mu$-{\it{essential}} if $\mu(\Env(\Sigma)) > 0.$ In words $\Sigma$ is $\mu$-essential if there is a positive probability that a $\mu$-random subgroup contains $\Sigma$. A subgroup is called {\it{essential}} if it is $\mu$-essential for some $\mu \in \IRS(\Gamma)$. We denote the collections of $\mu$-essential subgroups and of essential subgroups by
$$\Eg(\mu) := \left \{ \Sigma \in \Sub(\Gamma) \ | \ \mu(\Env(\Sigma)) > 0 \right \},\qquad \qquad \Eg(\Gamma) = \bigcup_{\mu \in \IRS(\Gamma)} \Eg(\mu).$$ 
 An element $\gamma \in \Gamma$ is essential essential if the cyclic group that it generates is essential. Note that the essential elements and subgroups of $\Gamma$ associated with an IRS are {\it{deterministic}} objects.
\end{definition}

Recall that a group $\Gamma$ is said to {\it{satisfy a property locally,}} if this property holds for every finitely generated subgroup of $\Gamma$. In particular this terminology is used in the following: 
\begin{definition}
Given a countable group $\Gamma$ and $\mu \in \IRS(\Gamma)$ we say that $\Delta \in \Sub(\Gamma)$ is $\mu$-{\it{locally essential}} if every finitely generated subgroup of $\Delta$ is $\mu$-essential. We say that $\Delta$ is {\it{locally essential}} if it is $\mu$-locally essential for some $\mu \in \IRS(\Gamma)$. 
\end{definition}
\begin{lemma} \label{lem:loc_ess} (The locally essential lemma)
Let $\Gamma$ be a countable group and $\mu \in \IRS(\Gamma)$. Then $\mu$-almost every $\Delta \in \Sub(\Gamma)$ is locally essential. 
\end{lemma}
\begin{proof}
 Let $H_1,H_2,\ldots \in \Sub(\Gamma)$ be an enumeration of the finitely generated subgroups of $\Gamma$ that are not $\mu$-essential. Then by definition $\mu(\Env(H_i)) = 0$ for every $i$, but
$$\left\{\Delta \in \Sub(\Gamma) \ | \ \Delta {\text{ is not locally essential}} \right\} = \bigcup_{i} \Env(H_i),$$
which has measure zero as a countable union of nullsets. 
\end{proof}
\begin{definition} \label{def:Rec}
We will say that a subgroup $\Delta \in \Sub(\Gamma)$ is {\it{recurrent}} if for every finitely generated subgroup $\Sigma < \Delta$ and for every group element $\gamma \in \Gamma$ the corresponding set of return times 
\begin{eqnarray*}
N(\Delta, \Env(\Sigma), \gamma) & = & \left \{n \in \N \ | \ \gamma^n \Delta \gamma^{-n} \in \Env(\Sigma) \right\} \\
& = &  \left \{n \in \N \ | \ \gamma^{-n} \Sigma \gamma^{n} < \Delta \right\} 
\end{eqnarray*}
is infinite. We will say that $\Delta$ is {\it{locally essential recurrent}} if it is both locally essential and recurrent. We will denote the collection of recurrent subgroups by $\Rec(\Gamma)$ and the collection of nontrivial recurrent subgroups by $\Rec^{0}(\Gamma) = \Rec(\Gamma) \setminus \{ \trivgp\}$. 
\end{definition}


\begin{corollary} \label{cor:loc_rec}
Let $\Gamma$ be a countable group and $\mu \in \IRS(\Gamma)$. Then $\mu$-almost every $\Delta \in \Sub(\Gamma)$ is recurrent. 
\end{corollary}
\begin{proof}
This basically follows directly from the locally essential lemma combined with Poincar\'{e} recurrence. We elaborate:

Fix $\mu \in \IRS(\Gamma)$. For a finitely generated subgroup $\Sigma \in \Sub(\Gamma)$ and a group element $\gamma \in \Gamma$ we denote by 
$$\Xi(\Sigma, \gamma) = \{\Delta \in \Env(\Sigma) \ | \ N(\Env(\Sigma),\Delta, \gamma) {\text{ is finite}}\},$$
the set of subgroups that fail to be recurrent due to this specific pair $(\Sigma, \gamma)$. Clearly $\Rec(\Gamma) = \Sub(\Gamma) \setminus \left( \bigcup \Xi(\Sigma, \gamma) \right)$ where the union ranges over all finitely generated subgroups $\Sigma$ and over all elements $\gamma \in \Gamma$. Since there are only countably many pairs $(\Sigma,\gamma) \in \Sub^{f.g.}(\Gamma) \times \Gamma$, it is enough to fix such a pair and show that $\mu(\Xi(\Sigma,\gamma))=0$. This is clear when $\Sigma$ fails to be $\mu$-essential. When $\mu(\Env(\Sigma))>0$ it follows from Poincar\'{e} recurrence. 
\end{proof}

\subsection{The stabilizer topology on a free group}
\label{sec:free}
In this section we sketch the existence proof of the stabilizer topology when $\Gamma = F_r$ is a nonabelian free group. This illustrates the main geometric idea of the argument while avoiding many of the technical details of the main proof. It will also illustrate the importance of the locally essential lemma. 
\begin{theorem} \label{thm:commutator}
Let $\Gamma$ be a countable group and $\Delta_1,\Delta_2 \in \Sub(\Gamma)$ two recurrent subgroups. Then for every pair of elements $\delta_i \in \Delta_i$
$$[\delta_1^{n_1},\delta_2^{n_2}] \in \Delta_1 \cap \Delta_2,$$
for infinitely many pairs $(n_1,n_2)\in \Z^2$. 
\end{theorem}
\begin{proof}
By recurrence, applied to the cyclic subgroup $\langle \delta_2 \rangle < \Delta_2$ we know that $\delta_1^{n_1} \delta_2 \delta_1^{-n_1} \in \Delta_2$ infinitely often and hence $[\delta_1^{n_1},\delta_2^{n_2}] \in \Delta_2$ for {\it{every}} $n_2$ and infinitely many values of $n_1$. An identical argument shows that $[\delta_1^{n_1},\delta_2^{n_2}] \in \Delta_1$ for every $n_1$ and infinitely many values of $n_2$ and the claim follows. 
\end{proof}
Recall that elements in a group $\Gamma$ are called {\it{independent}} if they freely generate a free group. 
\begin{lemma} \label{lem:independent}
Let $\{\Delta_1, \ldots, \Delta_J\} \subset \Sub(F_r)$ be non-abelian subgroups of the free group $F_r$. Then we can find $J$ independent elements $\delta_i \in \Delta_i$.
\end{lemma}
\begin{proof}
In a free group two elements commute if and only if they are contained in a cyclic subgroup. A nonabelian free group is clearly not a union of finitely many cyclic subgroups, so by induction we find a sequence of elements $\delta'_i \in \Delta_i$ that are pairwise non-commuting. Fix a free generating set $F_r = \langle S \rangle$ and consider the action of $F_r$ on its Cayley tree $T= \mathrm{Cay}(F_r,S)$. If we write $\delta'_i = \eta_i \theta_i \eta_i^{-1}$ where $\theta_i$ is cyclicly reduced then $\delta_i'$ acts on $T$ as a hyperbolic tree automorphism whose axis is the bi-infinite geodesic labeled periodically $\ldots \theta_i \theta_i \theta_i \ldots$ and passing through $\eta_i$. In particular these $J$ axes have $2J$ distinct end points in $\partial T$. Now by the standard ping-pong lemma we can find integers $n_j$ such that $\left \{ \delta_j = (\delta'_j)^{n_j} \ | \ 1 \le j \le J \right  \}$ are independent. Note that the groups $\Delta_j$ may very well be identical this does not change the details of the above proof.  
\end{proof}

Recall that a collection of subgroups $\Base \subset \Sub(\Gamma)$ forms a basis of identity neighborhoods of a topology if it satisfies the following properties (i) It is conjugation invariant (ii) It is a filter base, namely for every $\Delta_1,\Delta_2 \in \Base$ there exists some $\Delta_3 \in \Base$ with $\Delta_3 < \Delta_1 \cap \Delta_2$ (see  \cite[Proposition 1, TGIII.3]{Bourbaki:GT}). A collection of subgroups $\SB$ forms a sub-basis of identity neighborhoods if it is closed under conjugation. Such a sub-basis yields a basis consisting of all possible intersections of finitely many subgroups taken from $\SB$. The resulting topology will be discrete if and only if there are subgroups $\left \{ \Delta_1,\Delta_2, \ldots, \Delta_M \right \} \subset \SB$ with $\cap_{i=1}^M \Delta_i = \trivgp$. 
\begin{theorem} \label{thm:baby_case}
Let $F_r$ be a nonabelian free group. Then $\Rec^{0}(F_r)$ the nontrivial recurrent subgroups form a sub-basis of identity neighborhoods for a non-discrete group topology. This topology has the property that for every $\mu \in \IRS(F_r)$, $\mu$-almost every $\trivgp \ne \Delta \in \Sub(\Gamma)$ is open. \end{theorem}
\begin{proof}
The last statement follows directly from Corollary \ref{cor:loc_rec}, so all we have to prove is that given 
$\{\Delta_j \ | \ 1 \le j \le J\} \subset \Rec^{0}$ the intersection $\bigcap_{j=1}^J \Delta_j \ne \trivgp$. 

Recall that by Definition \ref{def:Rec} the groups $\Delta_j$ themselves are not allowed to be trivial. Let $\delta_j \in \Delta_j$ be independent elements as in Lemma \ref{lem:independent}. We assume by induction on $J$ that we have already found an element $w \in \bigcap_{j=1}^{J-1} \Delta_j$. We include in the induction hypothesis the requirement that $w \in \langle \delta_1,\delta_2, \ldots, \ldots, \delta_{J-1} \rangle$, and that when it is written as a reduced word in these free generators it assumes the form $w = \delta_1 w' \delta_{J-1}^{-1}$. Applying the induction hypothesis to $\Delta_J, \Delta_{J-1},\ldots, \Delta_2$ in this order we obtain another element $u \in \bigcap_{j=2}^{J} \Delta_j$ which can be written as a reduced word in $\delta_2,\ldots, \delta_J$ assuming the form $u = \delta_J u' \delta_2^{-1}$.

Now we can find $n,m \in \Z$ such that 
$$v:=[w^{n},u^{m}] \in \cap_{j =1}^J \Delta_j.$$
This element is automatically contained in $\Delta_j$ for every $2 \le j \le J-1$ and by Theorem \ref{thm:commutator} $(m,n)$ can be chosen in such a way that $v \in \Delta_1 \cap \Delta_J$ as well. Finally it is clear that $v \in \langle \delta_1, \delta_2,\ldots, \delta_J \rangle$ and that when it is written as a reduced word in these free generators it assumes the form $v = \delta_1 v' \delta_J^{-1}$. This completes the induction and, in particular shows that $v \ne e$, as required. \end{proof}

\section{Preliminary proofs}
\subsection{Covering of IRS}
Throughout this section $\Gamma$ is any countable group, $\mu \in \IRS(\Gamma)$, $I \subset \IRS(\Gamma)$ is a (Borel measurable) set of invariant random subgroups. We will say that a subgroup is $I$-essential if it is $\mu$-essential for some $\mu \in I$. Denote by $\Eg(I)$ the set of all $I$-essential subgroups. When $I$ is understood we will often omit it from the notation writing just $\Eg$. We will use letters such as $\Fg, \Hg \subset \Eg$ to denote Borel subsets of $\Sub(\Gamma)$; less formally we will refer to Borel subsets $\Fg,\Hg \subset \Eg$ as {\it{families of subgroups}}. 

\begin{definition} \label{def:cover}
We will say that {\it{A family of subgroups $\Fg$ covers $\mu$}}, if their envelopes cover $\Sub(\Gamma)$, up to a $\mu$-nullset. In other words if $\mu$-almost every $\Delta \in \Sub(\Gamma)$ contains some $\Sigma \in \Fg$. {\it{$\Fg$ covers $I$}} if it covers every $\mu \in I$.
\end{definition}
It becomes easier to cover if you use smaller subgroups. For example if  $\trivgp \in \Fg$  then $\Fg$ covers all of $\IRS(\Gamma)$. If $\Fg$ is the collection of all finitely generated subgroups then clearly $\Fg \setminus \{\trivgp\}$ still covers $\IRS^0(\Gamma)$ - the nontrivial invariant random subgroups (recall that we say an IRS is nontrivial if it is nontrivial almost surely). A major theme of the current paper is to take such {\it{trivial}} covers and refine them to covers by larger and more interesting subgroups. The notion of refinement of a cover is made explicit in the following definition:
\begin{definition} \label{def:refine_cov}
Assume that $\Fg$ covers $\mu$ as in the previous definition. We say that a family of essential subgroups $\Hg$ {\it{refines this cover}} and write $\Fg <_{\mu} \Hg$, if for every $\Sigma \in \Fg$ and for $\mu$-almost every $\Delta \in \Env(\Sigma)$ there exists $\Theta \in \Hg$ such that $\Sigma < \Theta < \Delta$. If $I \subset \IRS(\Gamma)$ is covered by $\Fg$ we will say that $\Hg$ {\it{refines this cover}} and write $\Fg <_I \Hg$ if $\Fg <_{\mu} \Hg, \ \ \forall \mu \in I$. 
\end{definition}
\begin{definition} \label{def:monotone}
We say that a family of essential subgroups $\Fg$ is {\it{monotone}} if it is closed under passing to larger essential subgroups. Namely $\Fg$ is monotone if, together with every $\Sigma \in \Fg$ the collection $\Fg$ contains all the groups in $\Env(\Sigma) \cap \Eg$. 
\end{definition}
For example, the collections of infinite essential subgroups is monotone while the collection of finite essential subgroups is usually not. The following lemma is useful in the construction of refined covers. 
\begin{lemma} \label{lem:refine}
Let $\mu \in \IRS(\Gamma)$, and $\Fg,\Hg \subset \Eg(\mu)$ two (Borel) families of  $\mu$-essential subgroups. Assume that $\Fg$ is monotone and covers $\mu$. If every $\Sigma \in \Fg$ is contained in some $\Theta \in \Hg$. Then $\mu$ is covered by $\Hg$ and $\Fg <_{\mu} \Hg$.
\end{lemma}
\begin{proof}
Assume, by way of contradiction that for some $\Xi \in \Fg$ the set 
$$Y = Y(\Xi) = \Env(\Xi) \setminus \left( \bigcup_{\Theta \in \Hg \cap \Env(\Xi)} \Env(\Theta) \right)$$ is of positive measure. Since $\Fg$ does cover we can find some $\Sigma \in \Fg$ such that $\mu(\Env(\Sigma) \cap Y) > 0$. Consider the essential subgroup 
$\Sigma' := \cap_{\Delta \in \Env(\Sigma) \cap Y} \Delta > \Sigma$. By definition $\Sigma' > \Sigma$ so by monotonicity $\Sigma' \in \Fg$. So that there exists some $\Theta \in \Hg$ with $\Theta > \Sigma'$. This implies that $\Env(\Theta) \subset \Env(\Sigma') = Y \cap \Env(\Sigma) \subset Y$  in contrast to the definition of $Y$. This shows that $\Fg <_{\mu} \Hg$. The fact that $\Hg$ covers follows upon taking $\Xi = \trivgp$ in the argument above.  \end{proof}

Assume that $\Delta = \Gamma_x \leftIRS \Gamma$ is realized as the stabilizer of a $\mu$-random point for some probability-measure preserving action $\Gamma \curvearrowright (X, \Bc, \mu)$ as in Example (\ref{eg:pmp}) of Subsection \ref{sec:eg_IRS}. For any subgroup $\Sigma < \Gamma$ we denote $\Fix(\Sigma) = \{x \in X \ | \ \sigma x = x \ \ \forall \sigma \in \sigma \}$ and for any subset $O \subset X$ we denote $\Gamma_O = \{ g \in \Gamma \ | \ gx = x {\text{ for almost all }} x \in O \}$. Then, up to nullsets, $\Env(\Sigma) = \{\Gamma_x \ | \ x \in \Fix(\Sigma)\}$ and the essential subgroups are exactly these subgroups $\Sigma < \Gamma$ such that $\mu(\Fix(\Sigma)) > 0$. An essential covering by the essential subgroups $\{ \Sigma_1,\Sigma_2, \ldots \}$ is just a covering (up to nullsets) of the form $X = \cup_i \Fix(\Sigma_i)$. 

\subsection{Projective dynamics} \label{sec:projective}
Our main theorem is reminiscent of the Tits alternative. Indeed the main part of the proof lies in the construction of the dense free subgroup $F$ from Theorem \ref{thm:one_ell}. We follow the geometric strategy put forth by Tits in \cite{Tits:alternative} and further developed in \cite{MS:first}, \cite{MS:Maximal}, \cite{BG:Topological_Tits} and many other papers. According to this strategy, the free group is constructed by first finding a projective representation exhibiting rich enough dynamics and then playing, ping-pong on the corresponding projective space. In particular the paper \cite{BG:Topological_Tits} describes how to go about this when the group $\Gamma$ fails to be finitely generated. In this case the topological field can no longer be taken to be a local field, but rather a countable extension of a local field - which is no longer locally compact. The corresponding projective space in this case is no longer compact, but it is still a bounded complete metric space and the dynamical techniques used in the ping pong game still work. Following \cite[Section 7.2]{GG:primitive} we give a short survey of the necessary projective dynamics that will be needed. We refer the reader to \cite[Section 3]{BG:Dense_Free}, \cite[Sections 3 and 6]{BG:Topological_Tits} for the proofs. Section 6 in the last mentioned paper deals explicitly with the non finitely generated setting. We recommend skipping the technical details in this section in a first reading and referring to them only when they become needed in the proof. 

Let $k$  be a local field $\left\| \cdot\right\| $ the standard norm on $k^{n}$, i.e. the
standard Euclidean norm if $k$ is Archimedean and $\left\| x\right\| =\max_{1\leq i\leq n}|x_{i}|$
where $x=\sum x_{i}e_{i}$ when $k$ is non-Archimedean and $(e_{1},\ldots ,e_{n})$ is the canonical
basis of $k^{n}$. Let $K$ be an extension field of finite or countable degree over $k$, by  \cite[XII, 4, Th. 4.1, p. 482]{Lang:algebra} the absolute value on $k$ extends to $K$. The norm on both fields extends in the usual way to $\Lambda ^{2}k^{n}$ (resp. $\Lambda^2K^n$). The \textit{standard metric} on $\mathbb{P}(K^{n})$ is defined by $d([v],[w])=\frac{\left\| v\wedge w\right\|
}{\left\| v\right\| \left\| w\right\|}$, where $[v]$ denotes the
projective point corresponding to $v\in K^n$. All our notation will refer to this metric, and in particular if $R \subset \PP(K^n)$ and $\epsilon > 0$ we denote the $\epsilon$ neighborhood of $R$ by $(R)_{\epsilon} = \{x \in \PP(K^n) \ | \ d(x,R) < \epsilon\}$. The projective space $\PP(k^n)$ is embedded as a compact subset of the complete metric space $\PP(K^n)$.  

Let $U < \PGL_n(K)$ be the subgroup preserving the standard norm and hence also the metric on $\PP(K^n)$ and $A^{+} = \{\operatorname{diag}(a_1,a_2, \ldots,a_n) \in \PGL_n(K) \ | \ |a_1| \ge |a_2| \ge \ldots \ge |a_n|\}$. The Cartan decomposition (see \cite[Section 7]{BrT:72}) is a product decomposition $\PGL_n(K) = UA^{+}U$, namely each $g \in \PGL_n(K)$ as a product $g = k_g a_g k'_g$ where $k_g,k'_g \in U$ and $a_g \in A^{+}$. This is an abuse of notation because only the $a_g$ component is uniquely determined by $g$, but we will be careful not to imply uniqueness of $k_g,k'_g$. A calculation shows that a diagonal element $a \in A^{+}$  acts as an $\left|a_1/a_n \right|^2$-Lipschitz transformation on $\PP(K^n)$ which immediately implies:
\begin{lemma} \label{lem:proj->lip} ( \cite[Lemma 3.1]{BG:Dense_Free})
Every projective transformation $g \in \PGL_n(K)$ is bi-Lipschitz on $\mathbb{P}(K^{n})$.
\end{lemma}

For $\epsilon \in (0,1)$, we call a projective transformation $g \in \PGL_{n}(K)$ {\it{$\epsilon$-contracting}} if there exist a point $v_{g} \in \PP(K^n),$ called an attracting point of $g,$ and a projective hyperplane $H_{g} < \PP(K^n)$, called a repelling hyperplane of $g$, such that 
$$g \left(\PP(K^n) \setminus (H_{g})_{\epsilon} \right) \subset \left(v_{g}\right)_{\epsilon}.$$ 
Namely $g$ maps the complement of the $\epsilon$-neighborhood of the repelling hyperplane into the $\epsilon $-ball around the attracting point. We say that $g$ is $\epsilon
${\it{-very contracting}} if both $g$ and $g^{-1}$ are $\epsilon $-contracting. A
projective transformation $g \in \PGL_{n}(K)$ is called $(r,\epsilon )${\it{-proximal}} for some $r>2\epsilon >0$, if it is $\epsilon$-contracting with respect to some attracting point $v_{g} \in \PP(K^n)$ and some repelling hyperplane $H_{g}$, such that $d(v_{g},H_{g}) \geq r$. The transformation $g$ is called $(r,\epsilon )${\it{-very proximal}} if both $g$ and $g^{-1}$ are $(r,\epsilon )$-proximal. Finally, $g$ is
simply called {\it{proximal}} (resp. {\it{very proximal}}) if it is $(r,\epsilon)$-proximal (resp. $(r,\epsilon)$-very proximal) for some $r>2 \epsilon >0$. The attracting point $v_{g}$ and repelling hyperplane $H_{g}$ of an $\epsilon$-contracting transformation are in general not uniquely defined. Still under rather mild conditions it is possible to find a canonical choice of $v_{g}$ and $H_{g}$ which are fixed by $g$ and all of its powers:

\begin{lemma}\label{fix}(\cite[Lemma 3.1]{BG:Topological_Tits})
Let $\epsilon \in (0,\frac{1}{4})$. There exist two constants $
c_{1},c_{2}\geq 1$ (depending only on the field $K$) such that if $g \in \PGL_n(K)$ is an $(r,\epsilon)$-proximal transformation with $r \geq c_{1}\epsilon $ then it must fix a unique point $\overline{v}_{g}$ inside its attracting neighborhood and a unique projective hyperplane $\overline{H}_{g}$ lying inside its repelling neighborhood\footnote{by this we mean that if $v,H$ are any couple of a pointed a hyperplane with $d(v,H)\geq r$ s.t. the complement of the $\epsilon$-neighborhood of $H$ is mapped under $g$ into the $\epsilon$-ball around $v$, then $\overline{v}_g$ lies inside the $\epsilon$-ball around $v$ and $\overline{H}_g$ lies inside the $\epsilon$-neighborhood around $H$}.
Moreover, if $r\geq c_{1}\epsilon ^{2/3}$, then the positive powers $g^{n}$, $n\geq 1$, are $(r-2\epsilon ,(c_{2}\epsilon)^{\frac{n}{3}})$-proximal transformations with respect to these same $\overline{v}_{g}$ and $\overline{H}_{g}$. 
\end{lemma}
\begin{proof}
The proof of this lemma in \cite[Section 3.4]{BG:Topological_Tits} is done for a local field $k$. Most of the proof carries over directly to the larger field $K$, except for the existence of a $g$-fixed hyperplane $\overline{H}_g \subset R(g)$. That part of the proof  uses a compactness argument, which does not work of the larger field $K$. 

We consider the action of the transpose $g^t$ on the projective space over the dual $\PP((K^n)^{*})$. Points are can be identified with hyperplanes in $\PP(K^n)$ and vice-versa so if we prove that $g^t$ is still $(r,\epsilon)$ contracting with the same constants the first part of the proof in Breulliard-Gelander's proof would apply and give rise to a fixed point in the dual projective space, corresponding to the desired fixed hyperplane. Given vectors $v,w \in K^n$ and linear functionals $f,h \in (K^n)^*$, all of them of norm one, we will use the following formulas
\begin{eqnarray} \label{eq}
d([v],[w]) & = & \norm{v \wedge w} \label{eqq1} \\
d([v],[\ker f]) & = & \left|f(v)\right|   \label{eqq2} \\ 
\Hd([\ker f], [\ker h]) & = & \max_{x \in \ker h} \frac{\left|f(x) \right|}{\norm{x}} = \norm{f \wedge h} = d_{\PP((K^n)^{*})} ([f],[h])   \label{eqq3}
\end{eqnarray}
Where $\Hd$ denotes the Hausdorff distance. All distances are measured in $\PP(K^n)$ except in the very last expression, where the contrary is explicitly noted. Equations (\ref{eqq1}),(\ref{eqq2}) appear in \cite[Section 3]{BG:Dense_Free}. The first and last equalities in (\ref{eqq3}) follow directly form (\ref{eqq2}) and (\ref{eqq1}) respectively. To check the middle equality in (\ref{eqq3}) let us choose a basis so that $h = (h_1, 0,0\ldots, 0), f = (f_1, f_2, \ldots, f_n)$. In this basis the inequality we wish to prove becomes
$$\max_{(x_2, \ldots, x_n) \in K^{n-1}} \frac{\left|f_2x_2 + f_3x_3 + \ldots f_nx_n \right|}{\norm{x}} = \norm{f_2 \hat{e_1} \wedge \hat{e_2} + f_3 \hat{e_1} \wedge \hat{e_3} + \ldots + f_n \hat{e_1} \wedge \hat{e_n}}.$$
Which is indeed true. 

Let $(H,v)$ be a hyperplane and unit vector in $K^n$ realizing the fact that $g$ is $(r,\epsilon)$-proximal. We claim that the transpose $g^t$ is $(r,\epsilon)$-proximal with respect to the dual hyperplane point pair $(L,f)$. Explicitly $f$ is a norm one linear functional with $\ker(f) = H$; and $L = \{l \in (K^{n})^{*} \ | \ l(v) = 0 \}$. The fact that $g$ is $\epsilon$-contracting with respect to $(H,v)$ translates, using Equations (\ref{eq}), into 
$$\left| f(u) \right| \ge \epsilon \norm{u}  \quad \Rightarrow \quad \norm{gu \wedge v} \le \epsilon \norm{g u}.$$
Assume that $\phi$ is any norm one functional with $d_{\PP((K^n)^*)}([\phi],[L]) > \epsilon$. Considering $v$ as a functional on $(K^n)^{*}$ with kernel $L$, the second line in (ref{eq}) gives $|\phi(v)| > \epsilon$. The same equation, now with $\phi$ considered as a functional on $K^n$ implies that $d_{\PP(K^n)}([\ker(\phi)],[v]) > \epsilon$. Since $g$ is by assumption $\epsilon$ contracting this means that 
$$\epsilon > \Hd(g^{-1} \ker(\phi), H) = \Hd(\ker(g^t \phi),H) =  \max_{x \in H} \left\{\frac{\left| g^t \phi (x) \right|}{\norm{x}} \right\} = d([g^t \phi], [f]).$$ Which is exactly the desired $\epsilon$-contraction. 

Note also that
$$d_{\PP((K^{n})^*)}([f],[L]) = d_{\PP((K^{n})^*)}([f],[\ker{v}]) = \left| f(v) \right| = d_{\PP(K^{n})}([\ker f],v) \ge r.$$
So that $g^t$ is indeed $(r,\epsilon)$ contracting with respect to $([L],[f])$. Applying the Breulliard-Gelander proof we obtain a unique $g^t$ fixed point $[\overline{\eta}] \in \left( [f] \right)_{\epsilon}$ which, under the additional assumptions on $(r,\epsilon)$ is also a common attracting point for all powers of $g^t$. We set $\overline{H}_g = \ker(\eta)$. Clearly this is a $g$ fixed subspace and it is easy to verify that it serves as a common repelling hyperplane for all the powers of $g$. \end{proof}

In what follows, whenever we add the article \textit{the} (or {\it the canonical}) to an
attracting point and repelling hyperplane of a proximal transformation $g$, we shall
mean these fixed point $\overline{v}_{g}$ and fixed hyperplane $\overline{H}
_{g}$ obtained in Lemma \ref{fix}. Moreover, when $r$ and $\epsilon$ are given, we shall denote by
$A(g),R(g)$ the $\epsilon$-neighborhoods of $\overline{v}_g,\overline{H}_g$ respectively. In some
cases, we shall specify different attracting and repelling sets for a proximal element $g$, denoting them by $\mathcal{A}(g),\mathcal{R}(g) \subset \PP(K^n)$ respectively. This means that
$$g \big(\PP(K^n)\setminus\mathcal{R}(g)\big)\subset\mathcal{A}(g). $$
If $g$ is very proximal and we say that
$\mathcal{A}(g),\mathcal{R}(g),\mathcal{A}(g^{-1}),\mathcal{R}(g^{-1})$ are specified attracting
and repelling sets for $g,g^{-1}$ then we shall always require additionally that
$$
 \mathcal{A}(g)\cap\big(\mathcal{R}(g)\cup\mathcal{A}(g^{-1})\big)=
 \mathcal{A}(g^{-1})\cap\big(\mathcal{R}(g^{-1})\cup\mathcal{A}(g)\big)=\emptyset.
$$

\medskip

Very proximal elements are important as the main ingredients for the construction of free groups using the famous ping-pong lemma:
\begin{lemma}\label{lem:ping-pong} (The ping-pong Lemma)
Suppose that $\{ g_i\}_{i\in I}\subset\PGL_n(K)$ is a set of very proximal elements, each associated with some given attracting and repelling sets for itself and for its inverse. Suppose that for any $i\neq j,~i,j\in I$ the attracting set of $g_i$ (resp. of $g_i^{-1}$) is disjoint from both the attracting and repelling sets of both $g_j$ and $g_j^{-1}$, then the $g_i$'s form an independent set, i.e. they freely generate a free group.
\end{lemma}
A set of elements which satisfy the condition of Lemma \ref{lem:ping-pong} with respect to some given attracting and repelling sets will be said to form a ping-pong set (or a ping-pong tuple).
\medskip

A novel geometric observation of the papers \cite{BG:Dense_Free,BG:Topological_Tits} is that an element $g \in \PGL_n(K)$  is $\epsilon$-contracting, if and only if it is $\epsilon'$-Lipschitz on an arbitrarily small open set. This equivalence comes with an explicit dependence between the two constants, as summarized in the lemma below:\begin{lemma}\label{lem:contracting=Lipschitz} (\cite[Proposition 3.3 and Lammas 3.4 and 3.5]{BG:Dense_Free}). 
There exists some constant $c$, depending only on the local field $k < K$,
such that for any $\epsilon \in (0,\frac{1}{4})$ and $d \in (0,1)$,
\begin{itemize}
\item if $g \in \PGL_n(K)$ is $(r,\epsilon)$-proximal for $r>c_1 \epsilon$,
then it is $c\frac{\epsilon^2}{d^2}$-Lipschitz outside the $d$-neighborhood of the repelling hyperplane, 
\item Conversely, if $g \in \PGL_n(K)$ is $\epsilon^2$-Lipschitz on some open neighborhood then it is $c \epsilon$-contracting.
\end{itemize}
Here $c_1$ is the constant given by Lemma \ref{fix} above.
\end{lemma}  
\begin{proof}
The proof for the case $K = k$ is given by the Proposition and two Lemmas indicated between brackets in the statement of the lemma are stated for a local field $k$. There the equivalence between the two properties is established via a third property, namely that  $\left|(a_2)_{g} /(a_1)_g\right| < c'' \epsilon^2$. 

It is routine to go over the proofs there and verify that they work in our more general situation. Basically there is only one thing to notice. In the proof that $\epsilon$-contraction implies $|a_2(g)/a_2(g)| < \epsilon^2/|\pi|$ (the converse implication of \cite[Proposition 3.3]{BG:Dense_Free}), the constant does depend on the field in a non-trivial way. Indeed $|\pi|$ here is the absolute value of the uniformizer of the (non-Archimedean in this case) local field. However when one passes to a larger field the the situation only gets better as the norm of the uniformizer can only grow. If ultimately the value group of $K$ is non-discrete one can do away with $|\pi|$ altogether (taking $|\pi| =1$ as it were).  
\end{proof}

\medskip
The above Lemma will enable us to obtain contracting elements - we just have to construct elements with good Lipschitz constants on an arbitrarily small open set. What we really need, in order to play ping-pong are very-proximal elements. To guarantee this we need to assume that our group is large enough.  
\begin{definition} \label{def:si}
A group $G \leq \PGL_n(K)$ is called {\it{irreducible}} if it does not stabilize any non-trivial projective subspace. It is called {\it{strongly irreducible}} if the following equivalent conditions hold
\begin{itemize}
\item It does not stabilize any finite union of projective hyperplanes,
\item Every finite index subgroup is irreducible,
\item The connected component of the identity in its Zariski closure is irreducible.
\end{itemize}
We will say that a projective representation $\rho: G \arrow \PGL_n(K)$ is {\it{reducible}} or {\it{strongly irreducible}} if the image of the representation has the same property. 
\end{definition}
The equivalence of the first two conditions is clear. The equivalence with the third property follows from the fact that fixing a projective subspace is a Zariski closed condition. It turns out that if a group is strongly irreducible and it contains contracting elements then the existence of very proximal elements is guaranteed.
\begin{lemma}\label{lem:contracting->very-proximal} (See \cite[Proposition 3.8 (ii) and (iii)]{BG:Dense_Free})
Suppose that $G\leq \PGL_n(K)$ is a group which acts strongly irreducibly on the projective space $\PP (K^n)$. Then there are
constants
$$
 \epsilon (G),r(G),c(G)>0
$$
such that if $g\in G$ is an $\epsilon$-contracting transformation for some $\epsilon <\epsilon (G)$ then for
some $f_1,f_2\in G$ the element $gf_1g^{-1}f_2$ is $(r(G),c(G)\epsilon)$-very proximal.
\end{lemma}
\noindent Combining the above two lemmas we obtain.
\begin{lemma}\label{lem:lip->very-proximal}
Suppose that $\Sigma \leq \PGL_n(K)$ is a group which acts strongly irreducibly. Then there are constants $\epsilon (\Sigma),r(\Sigma),c(\Sigma)>0$ such that for every $\lambda \in \PGL_n(K)$ which is locally $\epsilon$-Lipschitz for some $\epsilon <\epsilon (\Sigma)$ there exist $f_1, f_2\in \Sigma$ such that $\lambda f_1 \lambda^{-1} f_2$ is $(r(\Sigma),c(\Sigma)\sqrt{\epsilon})$-very proximal.
\end{lemma}

\noindent Finally, we will need the following two elementary lemmas. 
\begin{lemma} \label{lem:fix>triv}
A projective linear transformation $[B] \in \PGL_n(K)$ that fixes $n+1$ (projective) points in general position is trivial. 
\end{lemma} 
\begin{proof}
By definition, $n+1$ vectors in $K^n$ represent projective points in {\it{general position}} if any $i$ of them span an $i$-dimensional subspace as long as $i \le n$. If $\{v_0, v_1, \ldots,v_n\}$ are $n+1$ eigenvectors of $B$ in general position then by counting considerations there must be at least one eigenspace $V<K^n$ of dimension $l$ containing at least $l+1$ of these vectors. If the vectors are in general position then $l=n$, so $B$ has only one eigenvalue and is hence projectively trivial. \end{proof}
\begin{lemma} \label{lem:si_large_orb}
Assume that $\rho: \Gamma \arrow \PGL_n(K)$ is a strongly irreducible projective representation. Then every orbit $\rho(\Gamma)\overline{v}$ contains a set of $n+1$ points in general position. 
\end{lemma}
\begin{proof}
Without loss of generality we assume the field is algebraically closed. Let $\H = \overline{\rho(\Gamma)}^{Z}$ be the Zariski closure, $\H^{(0)}$ the connected component of the identity and $\Gamma^{(0)} = \rho^{-1} \left(\rho(\Gamma) \cap \H^{(0)}(K) \right)$. Strong irreducibility is equivalent to the irreducibility of $\rho \left(\Gamma^{(0)} \right)$, thus the $\rho(\Gamma^{(0)})$-orbit of every non trivial vector $v \in k^n$ contains a basis $B:=\{v = v_1,\ldots, v_n\}$. Let $V^i : = \operatorname{span} \{v_1,v_2, \ldots, \hat{v_i}, \ldots, v_n \}$ - the $i^{th}$ element missing. Assume by way of contradiction that $\rho(\Gamma^{(0)}) v \subset \cup_{i=1}^n V^i$.  Since taking a vector into a subspace is a Zariski closed condition we obtain, upon passing to the Zariski closure that
$$\H^{(0)}(K) = \cup_{i=1}^n H_i \qquad \qquad H_i = \{h \in \H^{(0)}(K) \ | \ h(v) \in V_i(K) \}$$
This is a union of Zariski closed sets, and since $\H^{(0)}$ is Zariski connected we have $H^{(0)}(K) = H_i$ for some $i$ contradicting irreducibility of $\H^{(0)}$. 
 \end{proof}

\section{Proof of the main theorem} 
In this chapter we reduce the proof of the Main Theorem \ref{thm:main_irs}, to our Main Technical Theorem \ref{thm:main_tec} concerning the existence of a certain cover. That theorem is is then proved in the following chapter.
\subsection{Reduction to the simple case} \label{sec:reduction_as}
At first we prove the Main Theorem \ref{thm:main_irs}, assuming the validity of its the simple version \ref{thm:one_ell}. 

The proof is preceded by two lemmas. As mentioned in Remark \ref{rem:ast_ammenable}, a group with a simple Zariski closure might be amenable. If it is not amenable though, it has no nontrivial normal amenable subgroups:
\begin{lemma} \label{lem:trivrad}
If $\Gamma < \GL_n(F)$ is a nonamenable linear group with a Zariski closure that is simple and has no finite normal subgroups, then $A = A(\Gamma) = \trivgp$. 
\end{lemma}
\begin{proof}
Let $\H = \overline{\Gamma}^{Z}$ be the Zariski closure. Since $\overline{A}^Z$ is normal in $\H$ our assumptions imply $\overline{A}^Z > \H^{(0)}$ whenever $A(\Gamma) \ne \trivgp$. Assume by contradiction that the former case holds. By Schur's theorem \cite[Theorem 4.9]{Weherfritz:linear_groups} every torsion linear group is locally finite and hence in particular amenable; since by assumption this is not the case for $\Gamma$ we can fix an element of infinite order $\gamma \in \Gamma$. The group $\Theta = \langle \gamma, A \rangle$, being amenable by cyclic, is still amenable and its Zariski closure still contains $\H^{(0)}$. The group $\Theta$ cannot contain a non-trivial normal solvable subgroup $S \lhd \Theta$. Indeed if $S$ were such a subgroup then $\left(\overline{S}^{Z} \right)^{(0)} \lhd \H^{(0)}$ would be a Zariski closed normal solvable subgroups, since the Zariski closure of a solvable subgroup is still solvable. Clearly there is no such group. Now, the Tits alternative \cite[Theorem 2]{Tits:alternative} implies that $\Theta$ is a locally finite group contrary to the existence of the infinite order element $\gamma$. 
\end{proof}
\begin{lemma} \label{lem:f.i.centralizer}
If $\Delta \leftIRS \Gamma$ is an IRS in a countable group $\Gamma$ that is almost surely finite then the centralizer $Z_{\Gamma}(\Delta)$ is of finite index in $\Gamma$ almost surely.
\end{lemma}
\begin{proof}
The collection of finite subgroups of $\Gamma$ is countable. Any ergodic measure that is supported on a countable set has finite orbits which in our case means that $\Delta$ has a finite conjugacy class almost surely. Thus $N_{\Gamma}(\Delta)$ is of finite index almost surely. The claim follows since $Z_{\Gamma}(\Delta)$ is the kernel of the action of $N_{\Gamma}(\Delta)$ on the finite group $\Delta$.
\end{proof}

\begin{proof} (Theorem \ref{thm:main_irs} follows from Theorem \ref{thm:one_ell})
 
Let $\eta_1: \Gamma \arrow \GL_n(\Omega)$ be an injective representation, realizing  our given group $\Gamma$ as a linear group over an algebraically closed field $\Omega$, let $\G_1 = \overline{\eta_1(\Gamma)}^{Z}$ be its Zariski closure which is by assumption connected. Replacing $\eta_1$ by $\eta_2 = \iota \compos \eta_1$ where  $\iota$ is the natural map to the semisimple quotient $\G_1/\Rad(\G_1)$ we obtain a new linear representation where $\G_2 = \overline{\eta_2(\Gamma)}^{Z}$ is semisimple. The map $\eta_2$ is no longer faithful but its kernel, being solvable and normal, is contained in the amenable radical. Moreover $\eta_2(A) < \eta_2(\Gamma)$ is a normal amenable subgroup.  

The Lie algebra $\gc := \Lie(G_2)$ decomposes as a direct sum of simple Lie algebras $\gc = \oplus_{\ell =1}^{L'} \hc_{\ell}$ which are invariant since $\G_2$ is connected. Thus we obtain a representation 
\begin{equation} \label{eqn:prod}
\prod_{\ell =1}^{L'} \chi_{\ell} = \Ad \compos \eta_2: \Gamma \arrow \prod_{\ell = 1}^{L'} \GL(\hc_{\ell}) < \GL(\gc),
\end{equation} 
where the representation $\chi_{\ell}: \Gamma \arrow \GL(\hc_{\ell})$ is defined by the above equation as the restriction to the $\ell^{th}$ invariant factor.  Let us denote by $\H_{\ell} = \overline{\chi_{\ell}(\Gamma)}^{Z}$ the Zariski closure of these representations; these are simple center free groups. By rearranging the order of the factors we may assume that $\chi_{\ell}(\Gamma)$ is nonamenable for every $1 \le \ell \le L$ and is amenable for every $L+1 \le \ell \le L'$. 

Let us set $A_{\ell}:= \ker(\chi_{\ell}) \lhd \Gamma$. The group $A^0 := \cap_{\ell =1}^{L} A_{\ell}$ is contained in the amenable radical $A(\Gamma)$, it is clearly normal and it is amenable as both $\ker(\Ad \circ \eta_2)|_{\Gamma^0}$ and 
$(\Ad \circ \eta_2)(\cap_{\ell =1}^{L} A_{\ell}) < \trivgp \times \ldots \times \trivgp \times \chi_{L+1}(\Gamma^0) \times \ldots \times \chi_{L'}(\Gamma^0)$ are. 
Let us set $A^{\ell} := \cap_{i=\ell+1}^{L} A_i$ for $0 \le \ell \le L$; thus $A^0 \lhd A^1 \lhd \ldots \lhd A^L = \Gamma$. Corresponding to this sequence of subgroups is the following stratification $\Sub(A^0) \subset \Sub(A^1) \subset \ldots \subset \Sub(A^L) =\Sub(\Gamma).$ Setting $S_0 = \Sub(A^0)$ and $S_{\ell} = \Sub(A^{\ell}) \setminus \Sub(A^{\ell -1}), \ 1 \le \ell \le L$ we can decompose $\Sub(\Gamma)$ as a disjoint union of the form $\Sub(\Gamma) = S_0 \sqcup S_1 \ldots \sqcup S_L$. Clearly $\chi_{\ell}(\Delta) \ne \trivgp$ for every $\Delta \in S_{\ell}$. 

Now fix $\mu \in \IRS(\Gamma)$ and decompose it according to this stratification by setting $a_{\ell} :=  \mu(S_{\ell}),$ and $\mu_{\ell}(B) := \mu(B \cap S_{\ell}) / a_{\ell},$ for every $0 \le \ell \le L$ and Borel subset $B \subset \Sub(\Gamma)$. Clearly $\mu = \sum_{\ell=0}^L a_{\ell} \mu_{\ell}$ as in the statement of the theorem. The groups $F_{\ell}$ for $1 \le \ell \le L$ are provided for upon application of theorem \ref{thm:one_ell} to the representations $\{\chi_{\ell} \ | \ 1 \le \ell \le L \}$. 

By definition $\mu_0$ is supported on $\Sub(A^0)$ so property {\bf{Amm}} is automatic. Similarly $\mu_{\ell}$ is supported on $S_{\ell}$ and $S_{\ell} \cap \Sub(\ker(\chi_{\ell})) = \emptyset$ so that $\mu_{\ell} \in \IRS^{\chi_{\ell}}(\Gamma)$ for every $1 \le \ell \le L$. Thus properties $\ell$-{\bf{Me-Dense}} and $\ell$-{\bf{Isom}} in theorem \ref{thm:main_irs} follow directly from the corresponding properties in Theorem \ref{thm:one_ell}. For property $\ell$-{\bf{Free}}. We can decompose $\mu_{\ell} = \mu_{\ell}^{\operatorname{NA}} + \mu_{\ell}^{\operatorname{A}}$ to an atomic part and a non-atomic part.  From Remark \ref{rem:free_follows} property $\mathbf{\chi}$-{\bf{Free}} follows directly form $\mathbf{\chi}${\bf{-Isom}} for the non-atomic part. Therefor we may assume that $\mu_{\ell}$ is completely atomic. This means that every ergodic component is supported on a finite orbit, or in other words that the IRS is supported on subgroups with finite index normalizers. So $\Delta \cap F$ also has a finite index normalizer in $F$, $\mu_{\ell}$ almost surely. Since $F$ itself is not finitely generated (by the proof of Theorem \ref{thm:one_ell}) it follows easily from properties of free groups that so is $\Delta \cap F$. 

It remains to be shown that it is possible to choose $L=1$ if and only if $\Gamma/A$ contains two nontrivial commuting almost normal subgroups. If $L \ge 2$ in the above construction then $\Gamma/A$ contains two commuting normal subgroups $\ker{\chi_1}$ and $\cap_{\ell = 2}^L \ker(\chi_\ell)$. 

Conversely let $\trivgp \ne M,N \lhd \Gamma$ be two almost normal subgroups such that  $M,N \not < A = A(\Gamma)$ but $[M,N] < A(\Gamma)$. Still assume by way of contradiction that the conclusion of the theorem holds with $L=1$. For convenience of notation we will set $F = F_1$. 

As explained in Example \ref{eg:fin}, both $M$ and $N$ appear with positive probability as instances of two invariant random subgroups $\mu_M,\mu_N \in \IRS(\Gamma)$. Thus by the statement of the theorem $F \cdot N = \Gamma$ and $O := F \cap M$ is a non-abelian free group. These two equations imply that $OA$ is an almost normal subgroup of $\Gamma/A$. Indeed let $\Gamma^0 = N_{\Gamma}(M)$, which is by assumption a finite index subgroup in $\Gamma$. If we set $F^0 = F \cap \Gamma^0$ then $F^0 N$ is of finite index in $\Gamma$. Clearly  $N < N_{\Gamma}(OA)$, and $F^0 < N_{\Gamma}(O)$ so that together $\Gamma^0A < N_{\Gamma}(OA)$. 

Now, applying the theorem again to the almost normal subgroup $OA$ we obtain, $\Gamma = F \cdot OA = FA$; but $A \cap F$ is an amenable normal subgroup of $F$ hence trivial. So that $\Gamma = F \ltimes A$ and $\Gamma/A \cong F$ is a free group which of course does not have two commuting almost normal subgroups, a contradiction. 
\end{proof}

\subsection{A good projective representation} \label{sec:good_lin_2}
Let $\chi: \Gamma \arrow \GL_n(F)$ be the simple representation given in Theorem \ref{thm:one_ell}. As explained in the beginning of Subsection \ref{sec:projective} the first step in proving Theorem \ref{thm:one_ell} is to fix a projective representation $\rho: \Gamma \arrow \PGL_n(K)$ over a topological field $K$ with enough elements that exhibit proximal dynamics on $\PP(K^n)$. This section is dedicated to this task. {\it{The notation fixed here will be used without further mention in the rest of the paper}}. 

When $\Gamma$ fails to be finitely generated $K$ is in general not a  local field. When this happens it is crucial to have a Zariski dense finitely generated subgroup $\Pi < \Gamma$ whose image under the representation will fall into the $\PGL_n(k)$ for some local subfield $k < K$. When $\Gamma$ is finitely generated one can take $k=K$ and $\Pi = \Gamma$. For the convenience of the reader we will try to indicate places where the argument in this case is more streamlined.  We start with a few lemmas. 
\begin{lemma} \label{lem:fg_subgroup}
Let $\Gamma < \GL_n(\Omega)$ be a non-amenable group with simple center free Zariski closure. Then there exists a non-amenable finitely generated subgroup $\Pi < \Gamma$ that has the same Zariski closure $\overline{\Pi}^{Z} = \overline{\Gamma}^{Z}$. Hence $\Pi$ is also non-amenable with a simple center free Zariski closure. 
\end{lemma}
\begin{proof}
Let $\H = \overline{\Gamma}^{Z}$ and let $\H^{(0)}$ be the Zariski connected component of the identity in $\H$. Since $[\H:\H^{(0)}] < \infty$ it is enough to find a finitely generated subgroup $\Pi<\Gamma$ such that $\overline{\Pi}^{Z} > \H^{(0)}$.  

Let $\{\gamma_1,\gamma_2,\ldots, \gamma_m\}$ be elements of $\Gamma$ generating a subgroup $\Pi'$ with the property that $\mathbf{I} = \overline{\Pi'}^{Z}$ is of maximal possible dimension. By Schur's theorem \cite[Theorem 4.9]{Weherfritz:linear_groups} every torsion linear group is locally finite and in particular amenable. Since $\Gamma$ is by assumption nonamenable it contains an element $\gamma_1 \in \Gamma$ of infinite order. The group  $\overline{\langle \gamma_1 \rangle}^{Z}$ is at least one dimensional and hence, also $\dim(\I) \ge 1$. Now by our maximality assumption, for every $\gamma \in \Gamma$ we have $\left(\overline{\langle \I, \gamma \rangle}^{Z}\right)^{(0)} = \I^{(0)}$ and in particular $\gamma$ normalizes $\I^{(0)}$. Since $\Gamma$ has a simple center free Zariski closure that implies $\I^{(0)} = \H^{(0)}$. 

A group is amenable if and only if all of its finitely generated subgroups are amenable. If every finitely generated group containing $\Pi$ were amenable this would imply that $\Gamma$ is amenable too, as the union of all these groups. Thus, after possibly replacing $\Pi$ with a larger finitely generated subgroup, we may always assume that $\Pi$ is non-amenable. This concludes the proof. 
\end{proof}
\begin{proposition} \label{prop:good_rep}
Let $\Pi$ be a finitely generated group and $\chi: \Pi \arrow \GL_n(f)$ be a representation with a simple center free Zariski closure. Then there is a number $r > 0$, a local field $k$  an embedding $f \hookrightarrow k$, an integer $n$, and a faithful strongly irreducible projective representation $\phi :\mathbb{\H}(k) \rightarrow \PGL_{n}(k)$ defined over $k$, such that for any $\epsilon \in \left(0,\frac{r}{2} \right)$ there is $g \in \Pi$ for which $\chi(g) \in \H^{(0)}(k)$ and $\phi \compos \chi (g)$ acts as an $(r,\epsilon )$-very proximal transformation on $\mathbb{P}(k^{n})$.
\end{proposition}
 \begin{proof} 
See \cite[Theorem 4.3]{BG:Topological_Tits} (also \cite[Theorem 7.6]{GG:primitive})  for a much more general statement. The faithfulness follows from the simplicity of the Zariski closure. \end{proof}

{\it{If $\Gamma$ is finitely generated}} we apply the above theorem with $f = F, \Pi = \Gamma$ and {\it{fix once and for all}} the local field $k = K$, the representation 
 $$\rho = \phi \compos \chi: \Gamma = \Pi \arrow \PGL_n(k) = \PGL_n(K),$$ and an element $g \in \Gamma$ such that $\rho(g) \in \left(\overline{\Gamma}^Z\right)^{(0)}$ and $\rho(g)$ is $(r,\epsilon)$-very proximal where $\epsilon$ is chosen so as to satisfy the conditions of Lemma \ref{fix}. Namely $r \ge c_1 \epsilon^{2/3}$, where $c_1$ is the constant, given in that lemma. We denote by $\overline{v}_{g^{\pm 1}}, \overline{H}_{g^{\pm 1}}$ the attracting and repelling points and hyperplanes associated to $g$ in that lemma. 

{\it{If $\Gamma$ is not finitely generated}} then we apply Lemma \ref{lem:fg_subgroup} and fix, once and for all, a nonamenable finitely generated subgroup $\Pi < \Gamma$ such that $\chi(\Pi)^{Z} = \chi(\Gamma)^{Z}$.  Let $f<F$ be the finitely generated subfield generated by the matrix coefficients of $\chi(\Pi)$. Applying Proposition \ref{prop:good_rep} to $\chi|_{\Pi}: \Pi \arrow \PGL_n(f)$ we obtain a number $r > 0$, a local field $k$  an embedding $f \hookrightarrow k$, an integer $n$, and a faithful strongly irreducible projective representation $\phi :\mathbb{\H}(k) \rightarrow \PGL_{n}(k)$ defined over $k$, such that for any $\epsilon \in \left(0,\frac{r}{2} \right)$ there is $g \in \Pi$ for which $\chi(g) \in \H^{(0)}(k)$ and $\phi \compos \chi (g)$ acts as an $(r,\epsilon )$-very proximal transformation on $\mathbb{P}(k^{n})$. Let  $\rho := \phi \compos \chi: \Pi \arrow \PGL_n(k)$. 

Setting $K = k \otimes_{f} F$, the absolute value on the local field $k$ extends to an absolute value on the extension field $K$ \cite[XII, 4, Th. 4.1, p. 482]{Lang:algebra}. The new field $K$ is not locally compact any more and if the original field $k$ was non-Archimedean then the corresponding extension of the discrete valuation to $K$ is a real valuation that need no longer be discrete. However $K$ is still a {\it{complete}} valued field. The representation $\phi$, defined as it is over $k$, extends to a representation which we will still call by the same name $\phi: \H(K) \arrow \PGL_n(K)$ which gives rise to an extension $\rho = \phi \circ \chi: \Gamma \arrow \PGL_n(K)$. 

\subsection{Definition of the stabilizer topology and the main technical theorem} \label{sec:stab_top}
Let $\chi: \Gamma \arrow \GL_n(F)$ be the group representation given in Theorem \ref{thm:one_ell}. We fix all the notation defined in Subsection \ref{sec:good_lin_2} and in particular the projective representation $\rho:\Gamma \arrow \PGL_n(K)$. The relevant class of invariant random subgroups appearing in Theorem \ref{thm:one_ell} is: 
$$\IRS^{\chiNF}(\Gamma) = \left\{\mu \in \IRS(\Gamma) \ \left | \ \chi(\Delta) \ne \trivgp {\text{ for $\mu$-almost every }}\Delta \in \Sub(\Gamma) \right. \right \}$$
Note that by the construction of the projective representation in the previous section $\ker(\rho) = \ker(\chi)$ so we can equivalently refer to the above class as $\IRS^{\rho}(\Gamma)$. 

Let $\Eg^{f.g.}$ be the collection of finitely generated $\IRS^{\chiNF}$-essential subgroups. Namely these subgroups that are finitely generated and $\mu$-essential for some $\mu \in \IRS^{\chiNF}$. Clearly this countable collection of subgroups covers $\IRS^{\chiNF}$. All of our main theorems will follow from the following theorem asserting the existence of a sufficiently good refinement of this cover.  The following theorem is the main technical theorem of the current paper, we defer its proof to the following chapter.   
\begin{theorem} \label{thm:main_tec}
There exists a collection of subgroups $\Hg = \{\Theta_1,\Theta_2, \ldots \} \subset \Eg^{f.g.}$ and a list of elements $\{f(p,q,\gamma) \ | \ p,q \in \N, \gamma \in \Gamma \}$ with the following properties.
\begin{enumerate}
\item \label{itm:cov} $\Eg^{f.g.} <_{\IRS^{\chiNF}} \Hg$, (see Definition \ref{def:refine_cov}).
\item \label{itm:pp} $\{f(p,q,\gamma) \ | \ p,q \in \N, \gamma \in \Gamma\}$ are independent.
\item \label{itm:double_cos} $f(p,q,\gamma) \in \Theta_q \gamma \Theta_p, \ \ \forall p,q \in \N, \gamma \in \Gamma$.
\end{enumerate}
\end{theorem}

Theorem \ref{thm:main_tec} enables us to define a basis of identity neighborhoods that will give rise to the stabilizer topology. As discussed right before Theorem \ref{thm:baby_case}, we can define a group topology by fixing a conjugation invariant sub-basis $\SB \subset \Sub(\Gamma)$ of identity neighborhoods consisting of subgroups. A basis for the topology near the identity will consist of finite intersections of the form $\Base = \left \{\cap_{i=1}^M \Delta_i \ | \ \Delta_i \in \SB \right \}$. Hence the resulting topology will be discrete if and only if there are subgroups $\left \{ \Delta_1,\Delta_2, \ldots, \Delta_M \right \} \subset \SB$ with $\cap_{i=1}^M \Delta_i = \trivgp$. 

\begin{definition} \label{def:H-top}
Let $\Gamma$ be a countable group. Given any collection of subgroups $\Hg \subset \Sub(\Gamma)$, the {\it{recurrent sub-basis associated with $\Hg$}} is the collection of all subgroups 
$$\SB_{\Hg} := \left \{\Delta \in \Rec(\Gamma) \ | \ \forall \gamma \in \Gamma, \ \exists \Theta \in \Hg {\text{ such that }} \Theta < \gamma \Delta \gamma^{-1} \right \}.$$
We denote by $\Base_{\Hg}$ and $\tau_{\Hg}$ the basis of identity neighborhoods and topology constructed from $\SB_{\Hg}$ respectively.  
\end{definition}

If $\Hg$ covers then this construction gives rise to a topology for which almost every subgroup is open.  
\begin{proposition} \label{prop:all_open} (Almost every subgroup is open)
Let $\Gamma$ be a countable group $I \subset \IRS(\Gamma)$ and $\Hg \subset \Sub(\Gamma)$ a family of subgroups that covers $I$. Then for every $\mu \in I$, $\mu$-almost every subgroup $\trivgp \ne \Delta \in \Sub(\Gamma)$ is open with respect to the topology $\tau_{\Hg}$. \end{proposition}
\begin{proof}
Fix $\mu \in I$, we will actually show that $\mu$-almost every subgroup is in $\SB_{\Hg}$. The proof that $\mu$ almost every subgroup is recurrent is identical to the one given in Corollary \ref{cor:loc_rec}. By our assumption that $\Hg$ covers $\mu$ so that $\mu  \left(\bigcup_{\Theta \in \Hg} \Env(\Theta) \right) = 1$. Consequently $\mu\left( \bigcap_{\gamma \in \Gamma} \gamma \left(\bigcup_{\Theta \in \Hg} \Env(\Theta) \right) \gamma^{-1} \right) = 1$, because $\Gamma$ is countable and $\mu$ is invariant. This completes the proof as we have shown that the two conditions for being in $\SB_{\Hg}$ are satisfied almost surely.
\end{proof}


\subsection{A proof that the main theorem follows from the existence of a good cover} \label{sec:proof_main}
In this section we prove Theorem \ref{thm:one_ell} assuming Theorem \ref{thm:main_tec}. The rest of the paper will be dedicated to the proof of the latter. 
\begin{proof} (of Theorem \ref{thm:one_ell} assuming Theorem \ref{thm:main_tec}).
Let $\chi: \Gamma \arrow \GL_n(F)$ be as in Theorem \ref{thm:one_ell} and apply Theorem \ref{thm:main_tec} which gives rise to the covering family $\Hg$ and the independent set $\{f(p,q,\gamma)\}$. Now:
\begin{itemize}
\item Construct the {\it{stabilizer topology}} by setting $\Me_{\chi} : = \tau_{\Hg}$,
\item Construct the free group by setting $F = \langle f(p,q,\gamma) \ | \ p,q \in \N ,\gamma \in \Gamma \rangle$. 
\end{itemize}
It follows from Proposition \ref{prop:all_open} that $\mu$-almost every subgroup is open in this topology for every $\mu \in \IRS^{\chi}(\Gamma)$. 

We first verify the equation appearing in property $\chi$-{\bf{Me-Dense}}:
 \begin{equation*} \label{eqn:density}
F \Delta = \Gamma, \qquad \forall \mu \in \IRS^{\chiNF}(\Gamma) {\text{ and for $\mu$-almost every }} \Delta \in \Sub(\Gamma).
\end{equation*}
 Fix $\mu \in \IRS^{\chiNF}$ and some $\gamma \in \Gamma$. By Property (\ref{itm:cov}) of Theorem \ref{thm:main_tec}, for $\mu$-almost every $\Delta \in \Sub(\Gamma)$ there is a $p \in \N$ such that $\Theta_p < \Delta$. Since $\mu$ is conjugation invariant, there is also almost surely some $q \in \N$ such that $\gamma^{-1} \Theta_q \gamma < \Delta$. Set $f = f(p,q,\gamma)$.  Now  by Property (\ref{itm:double_cos}) $f = \theta_q \gamma \theta_p$ for some $\theta_p \in \Theta_p, \theta_q \in \Theta_q$ so that we can write
$$\gamma = f \theta_p^{-1} \left( \gamma^{-1} \theta_q^{-1} \gamma \right) \in f \Theta_p \left( \gamma^{-1} \Theta_q \gamma \right) \subset f \Delta  \subset F \Delta.$$
Note that this {\it{does not show that $F$ is dense}} in the stabilizer topology. 

Proving $\chi$-{\bf{Non-Disc}} of the theorem will also show that this topology is non discrete. First note that for every $p$ the group $\Theta_p \cap F$ is a nonabelian free group as it contains the free group $\langle f(p,p,\theta) \ | \ \theta \in \theta_p \rangle < \Theta \cap F$. To proceed we have to show that for every every basic open set of the form $\bigcap_{j=1}^J \Delta_j$ where  $\{\Delta_1, \ldots, \Delta_J\} \subset \SB$ contains a nontrivial element of $F$. This is proved in precisely the same way as Theorem \ref{thm:baby_case}, only we have to take care that the elements $\delta_j$ appearing in that proof satisfy $\delta_j \in \Delta_j \cap F$. 

Let $\mu_1,\mu_2 \in \IRS^{\chi}(\Gamma)$. By Proposition \ref{prop:all_open} $\mu_i$-almost every subgroup of $\Gamma$ is open. Thus for $\mu_1 \times \mu_2$ almost every pair $(\Delta_1, \Delta_2) \in \Sub(\Gamma)^2$ both subgroups are open and by the previous paragraph, so is their intersection $\Delta_1 \cap \Delta_2$. This immediately implies that $\mu_1 \cap \mu_2 \in \IRS^{\chi}(\Gamma)$, where $\mu_1 \cap \mu_2$ is the intersection IRS defined in Subsection \ref{sec:ind_res}. 

To establish property $\mathbf{\chi-Isom}$ we have to show that the restriction map is an isomorphism. 
\begin{eqnarray*}
\Phi: (\Sub(\Gamma),\mu) & \arrow & (\Sub(F), \mu|_{F}) \\
\Delta & \mapsto & \Delta \cap F
\end{eqnarray*}
This is clearly a measure preserving $F$-invariant surjective map. By Souslin's theorem \cite[Corollary 15.2]{Kechris:calassical_dst} it is enough to show that $\Phi$ is essentially injective on $\Sub(\Gamma)$. In other words we wish to show that  $\Phi(\Delta) \ne \Phi(\Delta')$ for $\mu \times \mu$ almost every $(\Delta,\Delta') \in \Sub(\Gamma)^2$ such that $\Delta \ne \Delta'$. Applying Theorem \ref{thm:main_tec} (\ref{itm:cov}), we conclude that $\mu \times \mu$-almost surely there is some $\Theta \in \Hg$ with $\Theta < \Delta \cap \Delta'$. Since by assumption $\Delta \ne \Delta'$, after possibly switching the roles of $\Delta, \Delta'$ we can find an element $\delta \in \Delta \setminus \Delta'.$ Repeating the argument from the proof for property $\chi$-{\bf{Me-Dense}}, we can find a non-trivial element of $F$ in the coset $$f \in F \cap  \delta \Theta \subset F \cap \delta(\Delta \cap \Delta') = (F \cap \Delta) \cap (F \cap \delta \Delta').$$ Since two cosets of the same group are disjoint $f \in (F \cap \Delta) \setminus (F \cap \Delta') = \Phi(\Delta) \setminus \Phi(\Delta')$; establishing the injectivity. 
\end{proof}

\begin{note}
The dynamical interpretation of the argument used in the proof of property $\chi$-{\bf{Me-Dense}}, that was used twice in the above proof is as follows. By Lemma \ref{lem:IRS} every $\mu \in \IRS^{\chiNF}$ arises as the point stabilizer of a probability-measure preserving action $\Gamma \curvearrowright (X,\Bc,\mu)$.  Since $\Hg$ covers we can find, for every $\gamma \in \Gamma$ and for almost every $x \in X$, some $p,q$ such that $\Gamma_x > \Theta_p$ and $\Gamma_{\gamma x} > \Theta_q$. Hence $\gamma x = \theta_q \gamma \theta_p x, \ \forall \theta_p \in \Theta_p, \theta_q \in \Theta_q$. Theorem \ref{thm:main_tec} above provides a free subgroup with generators in each and every one of these double cosets. Thus for every probability-measure preserving action of $\Gamma$ with the property that $\chi(\Gamma_x) \ne \trivgp$ almost surely, $\Gamma$ and $F$ have the same orbits almost surely. \end{note}

\section{The construction of a good cover}
Finally we arrive at the proof of Theorem \ref{thm:main_tec}. The proof proceeds in a few steps in which we construct successive refinements for our cover of the collection $\IRS^{\chiNF}$. 

\subsection{A cover by infinite subgroups} \label{sec:essinf}
Our first goal is to show that $\IRS^{\chiNF}(\Gamma)$ admits a cover by essential subgroups whose $\rho$-image is infinite. This now follows as a direct corollary of the Bader-Duchesne-L\'{e}cureux Theorem \cite{BDL:amenable_irs}. In this section we show how. The following two sections are dedicated to my original construction of such a cover, which is geometric and of independent interest. 

Fix $\mu \in \IRS^{\chiNF}(\Gamma)$ and assume by way of contradiction that there is no such essential cover. In this case we shall further assume that $\left| \rho(H) \right| < \infty$ for every $\mu$-essential subgroup - just by conditioning on the event: $$\left \{\Delta \in \Sub(\Gamma) \ | \ \left| \rho(H) \right| < \infty, \ \forall H \in \Eg(\mu) \cap \Sub(\Delta) \right\}$$ which we assumed has positive probability. The locally essential lemma \ref{lem:loc_ess} now implies that $\rho(\Delta)$ is locally finite, and in particular amenable, $\mu$-almost surely. Now by the Bader-Duchesne-L\'{e}cureux theorem the IRS is contained in the amenable radical $A(\Gamma)$ almost surely, and the latter is trivial by Lemma \ref {lem:trivrad}. 

\subsection{A geometric construction of an infinite cover}
In this section we assume either that $\chr(F)=0$ or that $\Gamma$ is finitely generated. Countable, non finitely generated groups in positive characteristic will be treated separately in the next section. 

As in the previous section it is enough to prove that the only $\mu \in \IRS^{\chiNF}(\Gamma)$ supported on locally finite subgroups of $\Gamma$ is the trivial IRS. Given such $\mu$, it follows from the theory of linear groups that there exists a number $M$ such that $\rho(\Delta)$ admits a subgroup $\Lambda(\Delta)$ of index at most $M$ that fixes a point $s(\Delta) \in \PP(k^n)$, almost surely. Indeed if $\Gamma$ is linear in characteristic zero then every locally finite subgroup is virtually abelian, with a bound on the index of the abelian subgroup by Schur's theorem \cite[Corollary 9.4]{Weherfritz:linear_groups} and the claim follows from Mal'cev's Theorem \cite[Theorem 3.6]{Weherfritz:linear_groups}. If $\Gamma$ is finitely generated it follows from \cite[Corlloary 4.8]{Weherfritz:linear_groups} that there is a bounded index subgroup consisting of unipotent elements, which is therefore unipotent. 

Let $g \in \Gamma$ be any element whose $\rho$ image is very proximal and also satisfies the condition $r > c_1 \epsilon^{2/3}$ appearing in Lemma \ref{fix}. The existence of one such element is guaranteed at the end of Subsection \ref{sec:good_lin_2}. Adhering to the notation set in Subsection \ref{sec:projective} we donate by $\overline{v}_g, \overline{H}_g, \overline{v}_{g^{-1}}, \overline{H}_{g^{-1}}$ the canonical attracting points and repelling hyperplanes of $g,g^{-1}$ respectively.  

We claim that $\rho(\gamma)$ fixes $\overline{v}_g$ for every essential element $\gamma \in \Gamma$. Indeed, by the definition of an essential element, $\Env(\gamma) := \left\{\Delta \in \Sub(\Gamma) \ | \ \gamma \in \Delta\right \}$ is a set of positive measure. Thus by Poincar\'{e} recurrence, for almost all $\Delta \in \Env(\gamma)$ the sequence of return times 
$$N(\Delta, \Env(\gamma)) = \left \{n_k \ | \ \ g^{n_k} \Delta g^{-n_k} \in \Env(\gamma) \right \} = \left \{n_k \ | \ g^{-n_k} \gamma g^{n_k} \in \Delta \right \},$$ 
is infinite. Fix such a recurrent point $\Delta \in \Env(\gamma)$ which is, at the same time, locally finite. Many such exist because both properties hold with probability one - the first by Poincar\'{e} recurrence, the second by the locally essential lemma \ref{lem:loc_ess}. 

Now we follow the dynamics of the same elements on the projective space. Let $\Lambda = \Lambda(\Delta) < \Delta$ be the finite index subgroup of $\Delta$ fixing the point $, s = s(\Delta) \in \PP(k^n)$ as constructed in the beginning of this section. Let $\Omega := \Delta \cdot s$ be the (finite) orbit of this point. Assume first that $\Omega \cap \overline{H}_g =  \emptyset$. In this case $\rho(g^n) \omega \stackrel{{{n \arrow \infty}}}{\longrightarrow} \overline{v}_g$ for every $\omega \in \Omega$ and consequently 

\begin{eqnarray*}
\rho(\gamma)(\overline{v}_g) & = & \rho(\gamma) \left(\lim_{k \arrow \infty} \rho \left(g^{n_k} \right) s \right)  \\
& = & \lim_{k \arrow \infty} \rho \left(g^{n_k} \left( g^{-n_k} \gamma g^{n_k} \right) \right) s \in \lim_{k \arrow \infty}  \rho \left( g^{n_k}\right) \Omega =
\left \{ \overline{v}_g \right \}. 
\end{eqnarray*}

Consider now the general case when $\Omega \cap \overline{H}_g \ne \emptyset$.  In the next two paragraphs we will exhibit a sequence of very proximal elements $g_n$ who satisfy $\Omega \cap \overline{H}_{g_n} = \emptyset$, and $\lim_{n \arrow \infty} \overline{v}_{g_n} = \overline{v}_g$. By the previous paragraph the essential element $\gamma$ fixes $\overline{v}_{g_n}$ and since the fixed point set of $\rho(\gamma)$ is closed we conclude that $\rho(\gamma) \overline{v}_g = \overline{v}_g$. As desired. 

We can find an element $\sigma \in \Gamma$ such that the following conditions hold:
\begin{itemize}
\item $\rho(\sigma) (\Omega) \cap \overline{H}_g = \emptyset$.
\item $\rho(\sigma) \overline{v}_g \not \in \overline{H}_g$.
\item $\rho(\sigma^{-1}) \overline{v}_{g^{-1}} \not \in \overline{H}_{g^{-1}}$.
\end{itemize}
 Indeed we will even find such an element in $\Gamma^{(0)}$ - the intersection of $\Gamma$ with the connected component of its Zariski closure. The strong irreducibility of $\rho(\Gamma)$ is equivalent to the irreducibility of $\rho(\Gamma^{(0)})$ (see Definition \ref{def:si}). Thus for every $\omega \in \PP(K^n)$ the collection $D_{\omega} := \{\sigma \in \Gamma^{(0)} \ | \ \rho(\sigma) \cdot \omega \not \in \overline{H}_g \}$ is non empty and Zariski open. Since $\Gamma^{(0)}$ is Zariski connected the sets $D_{\omega}$ are open and dense. A similar dense open set can be constructed for the last condition $E := \{\sigma \in \Gamma^{(0)} \ : \ \rho(\sigma^{-1}) \overline{v}_{g^{-1}} \not \in \overline{H}_{g^{-1}} \}$. Our desired element can be chosen as any element in the intersection $\sigma \in \left(\cap_{\omega \in \Omega} D_{\omega} \right) \cap D_{\overline{v}_g} \cap E $. Having chosen such a $\sigma$ we let $d$ be the minimal distance attained in all the above relations, namely: 
 $$d \left(\rho(\sigma) \omega, \overline{H}_g \right) > d \ \forall \omega \in \Omega, \ d \left(\rho(\sigma) \overline{v}_g, \overline{H}_{g} \right) > d, d \left(\rho(\sigma^{-1}) \overline{v}_{g^{-1}}, \overline{H}_{g^{-1}} \right) > d.$$ 

Now consider the sequence of elements $g_n:= g^n \sigma$. If $\mathcal{A}(g^n)$, $\mathcal{R}(g^n)$, $\mathcal{A}(g^{-n})$, $\mathcal{R}(g^{-n})$ are attracting and replying neighborhoods for $\rho(g^n)$ then $\mathcal{A}(g_n) = \mathcal{A}(g^n)$, $\mathcal{R}(g_n) = \rho(\sigma^{-1}) \left(\mathcal{R}(g^n)\right)$, $\mathcal{A}(g_n^{-1}) = \rho(\sigma^{-1}) \left(\mathcal{A}(g^{-n})\right)$, $\mathcal{R}(g_n^{-1}) = \mathcal{R}(g^{-n})$ will be attracting and repelling neighborhoods for $g_n$. For large enough values of $n$ we can assume that the original neighborhoods are arbitrarily small $\mathcal{A}(g^n)\subset (\overline{v}_{g^n})_{\epsilon} = (\overline{v}_g)_{\epsilon}, \mathcal{R}(g^n) \subset (\overline{H}_g)_{\epsilon},$ etc'. Where $(\cdot)_{\epsilon}$ stands for the $\epsilon$-neighborhood. In particular if $d$ is the bound constructed in the previous paragraph and we choose $n$ large enough so that $d > \epsilon/2$. The above inequalities will imply immediately that 
$$Ac(g_n) \cap \Rc(g_n) = \Ac(g_n^{-1}) \cap \Rc(g_n^{-1}) = \Rc(g_n) \cap \Omega = \emptyset$$
The first two ensure that $\rho(g_n)$ is again very proximal. The last one shows that  $\Omega \cap \mathcal{A}(g_n) = \emptyset$ for every $n \ge N$. Thus as mentioned above we have $\rho(\lambda) \overline{v}_{g_n} = \overline{v}_{g_n}$ and passing to the limit this completes the proof that $\overline{v}_g$ is fixed by all essential elements. 

Since the representation $\rho: \Gamma \arrow \GL_n(K)$ is strongly irreducible Lemma \ref{lem:si_large_orb} yields $n+1$ points in general position in the orbit $\left \{\gamma_0 \overline{v}_g, \gamma_1 \overline{v}_g\ldots, \gamma_{n} \overline{v}_g \right \}$. Since $\gamma_i \overline{v}_g$ is the attracting fixed point for the very proximal element $\gamma_i g \gamma_i ^{-1}$, the above proof implies that all of these $n+1$ points are fixed by every essential element $\gamma \in \Gamma$. By Lemma \ref{lem:fix>triv} $\rho(\gamma)$ is trivial for every essential element. Finally the locally essential lemma \ref{lem:loc_ess} implies that $\rho(\Delta) = \trivgp$ for $\mu$-almost every $\Delta$ contradicting our assumption $\mu \in \IRS^{\chiNF}$. 

\subsection{Finitely generated groups in positive characteristic} \label{sec:countable_p}
Here we treat the case where the group $\Gamma$ is linear in positive characteristic and non finitely generated. Just as in the previous section it is enough to prove that the only $\mu \in \IRS^{\chiNF}(\Gamma)$ supported on locally finite subgroups of $\Gamma$ is the trivial IRS.  

The geometric strategy employed in the previous section to show that such IRS cannot exist is quite general. Loosely speaking two geometric properties of the action on the projective plane were used: 
\begin{enumerate}
\item \label{itm:LFP} Every locally finite subgroup fixes a point (or a finite number thereof)
\item \label{itm:enough_prox} There are enough proximal elements
\end{enumerate}
The problem is that when $\Gamma$ is linear over a field of positive characteristic and, at the same time, fails to be finitely generated the action on the projective space $\PP(K^n)$ no longer satisfies condition (\ref{itm:LFP}) above. For example if $F < K$ is a locally finite field then $\PGL_n(F) < \PGL_n(K)$ is locally finite, but does not fix any projective point in $\PP(K^n)$. We assumed explicitly that $\rho(\Gamma)$ is not an amenable group and in particular it cannot be of this form. Still $\rho(\Gamma)$ may contain many such locally finite subgroups and we have to show that the IRS is not supported on these.  

We will use the action of $\PGL_n(K)$ on its (affine) Bruhat-Tits building $X=X(K^n)$ and show that $X$ does satisfy the desired conditions. Note that in this generality the building $X$ is neither locally finite nor simplicial. This is due, respectively, to the facts that the residue field might no longer be finite and the value group might fail to be discrete. Bruhat-Tits buildings in this generality were treated by Brhuat-Tits \cite{BrT:72,BrT:84} but a geometric axiomatization of such affine buildings was given only later in the thesis of Anne Parreau \cite{Par:thesis}. 

We will follow the notation in the excellent survey paper \cite[Section 1]{RTW:survey}. In particular we set $V = K^n$ and denote by $\mathcal{N}(V,K)^{\operatorname{diag}}$ the set of all diagonalizable non-Archimedean norms on $V$ (these are discussed in some detail in \cite{Weil:basic_number_theory}) and we identify the building $X = \mathcal{X}(V,K) = \mathcal{N}(V,K)^{\operatorname{diag}} / \sim$ with the collection of all such norms modulo homothety $\norm{\cdot} \sim c \norm{\cdot}, \  c>0$. A norm is diagonalized by a basis $\underline{e} = \{e_1,\ldots, e_n\}$ of $V,$ if it is of the form $\norm{\sum a_i e_i}_{\underline{c}} = \max_i \{e^{c_i} |a_i| \}$ for some vector $\underline{c} = (c_1, \ldots,c_n)$. Thus the collection of all norms diagonalized by a given basis is identified with $\{\norm{\cdot}_{\underline{c}} \ | \ \underline{c} \in \R^n\} \cong \R^n$. The image of these Euclidean spaces in $\mathcal{X}(V,K)$ are the apartments of the building, they are $n-1$ dimensional spaces of the form $A_{\underline{e}} = \R^n/\langle(1,1,\ldots,1) \rangle$. The important features for us are that the building $X$ is still a $\CAT(0)$ space (see \cite[Section 1.1.3]{RTW:survey}) and it still has a finite dimensional Tits boundary.  In fact, since the base field is complete, the boundary $\partial X$ is naturally identified with the spherical building of $\PGL_n(K)$. Recall that the zero skeleton $\partial X^0$ of the spherical building consists of non-trivial proper subspaces of $V$. The higher dimensional simplexes correspond to flags of such subspaces. We will say that a vertex of $\partial X$ is of type $i$ if it corresponds to an $i$-dimensional subspace. 

%

These conditions are important because they are exactly these needed for the following theorem of Caprace-Monod.
\begin{proposition} (\cite[Corollary 3.4]{CM1}) \label{prop:CM}
Let $X$ be a $\CAT(0)$ space with finite dimensional Tits boundary. Then every locally finite group acting on $X$ fixes a point in $X \sqcup \partial X$.  
\end{proposition} 
Thus we may assume that almost every $\Delta \in \Sub(\Gamma)$ fixes some point $s(\Delta) \in X \sqcup \partial X$. Since the action of $\PGL_n(K)$ on $\partial X$ preserves types we may, and shall, assume that whenever $s(\Delta) \in \partial X$ then it is actually a vertex in the spherical building, of type $i(\Delta)$.

Let us write $\mu$ as a convex combination of IRS with disjoint supports as follows $\mu = c_X \mu_X + c_{1} \mu_{1} + c_{H} \mu_{H}$ where:
\begin{itemize}
\item $\mu_X$-almost every $\Delta \in \Sub(\Gamma)$ has a fixed point inside $X$. 
\item $\mu_{1}$-almost every $\Delta \in \Sub(\Gamma)$ does not have a fixed point in $X$ but it does fix a vertex of type $1$ in $\partial X$. 
\item $\mu_{H}$-almost every $\Delta \in \Sub(\Gamma)$ fixes neither a point in $X$ nor a vertex of type $1$ in $\partial X$. So that $s(\Delta) \in \partial X$ is a vertex of higher type. 
\end{itemize}
For $\mu_X$ almost very $\Delta \in \Sub(\Gamma)$ we insist that $s(\Delta) \in X$. Similarly for $\mu_1$ almost every $\Delta$ we take $s(\Delta)$ to be a vertex of type $1$. We will treat each of these measures separately. Note that only $\mu_X$ can be trivial as the other two give measure zero to the trivial group. Therefore what we have to prove is that $\mu_X (\{\trivgp\}) = 1$ and that the other two measures cannot exist and hence $c_X = 1, c_1 = c_H = 0$. The measure $\mu_1$ is easiest, $\mu_1$-almost every $\Delta \in \Sub(\Gamma)$ fixes a point in $\PP(V)$ and we argue exactly as in the previous section. 

For $\mu_X$ the argument is very similar. Let $g \in \Gamma$ be an element whose $\rho$ image is $(r,\epsilon)$-very proximal with attracting point $\overline{v}_g$ and repelling hyperplane $\overline{H}_g$ in the projective plane $\PP(V)$. Let $\tilde{v}_g = \langle v_g \rangle \in \partial X$ be the associated $g$ fixed point in the spherical building $\partial X$. $\rho(g)$ will act as a hyperbolic isometry on the building $X$. Just as in the previous section it would be enough to show that $\tilde{v}_g$, and hence also $\overline{v}_g$ is fixed by every $\mu_X$-essential element $\gamma \in \Gamma$. 

Pick $\Delta \in \Env(\gamma)$ that is simultaneously locally finite and recurrent in the sense that $g^{-n} \gamma g^{n} \in \Delta$ infinitely often. Let $s = s(\Delta) \in X$ be the corresponding fixed point. The element $\rho(g)$ acts as a hyperbolic element on $X$ with attracting point $\tilde{v}_g \in \partial X$. Thus for every sequence $n_i \in N(\Delta, \Env(\langle \gamma \rangle), g)$ the sequence $\rho(g^{n_i}) s$ consists of $\rho(\Delta)$ fixed points converging to the attracting point 
$\rho(g^{n_i}) s \stackrel{i \arrow \infty}{\arrow} \tilde{v}_g$. We conclude using the fact that set of $\rho(\gamma)$ fixed points is closed.

We turn to $\mu_H$, which is the hardest of the three. The following Lemma shows that $\mu_H$ cannot be supported on totally reducible subgroups. We first recall the definition.
\begin{definition} \label{def:tot_red}
A subspace of $V$ is called {\it{totally reducible}} if it decomposes as a direct sum of irreducible $\rho(\Delta)$ modules. We say that $\rho(\Delta)$ is totally reducible if $V$ itself is totally reducible as a $\rho(\Delta)$ module.
\end{definition}
\begin{lemma} \label{lem:tot_red->fp}
Let $K$ be a complete valued field, $\Delta$ a locally finite group and $\rho: \Delta \arrow \PGL_n(K)$ a totally reducible representation, all of whose irreducible components are of dimension $\ge 2$. Then $\rho(\Delta)$ fixes a point in the affine Bruhat-Tits building $X$ of $\PGL_n(K)$. 
\end{lemma}
\begin{proof}
Assume that $V = W_1 \oplus W_2 \oplus \ldots \oplus W_l$ is the decomposition into $\rho(\Delta)$-irreducible representations. Since each $W_i$ is of dimension at least two we consider the affine building $Y = Y_1 \times Y_2 \times \ldots \times Y_l$ corresponding to the group $\PGL(W_1) \times \PGL(W_2) \times \ldots \times \PGL(W_l)$. The boundary of this product is the spherical join of the boundaries of the individual factors $\partial Y = \partial Y_1 * \partial Y_2 * \ldots * \partial Y_l$, see \cite[Definition 5.13, Corollary 9.11]{BH:Metric_Spaces}. $\partial Y$ embeds as a sub-building of $\partial X$. Since the group $\rho(\Delta)$ preserves the direct sum decomposition, it acts on this building. Being locally finite as it is, Proposition \ref{prop:CM} implies that $\rho(\Delta)$ has a fixed point in $Y \cup \partial Y$. If the fixed point is on the boundary we may assume again that it is a vertex; but $\partial Y^{0} = \partial Y_1^{0} \sqcup \ldots \sqcup \partial Y_l^{0}$ so that in fact we have a fixed vertex in $\partial Y_i$ for some $i$. This corresponds to a proper invariant subspace of $W_i$ which is impossible since the action of $\rho(\Delta)$ on $W_i$ is irreducible. So there is a $\rho(\Delta)$ fixed point $y = (y_1,y_2, \ldots, y_l) \in Y$. Each $y_i$ corresponds to a homothety class of norms, let $\norm{\cdot}_i$ be a norm on $W_i$ representing this homothety class. Consider the subspace of $X$ consisting of all homothety classes of norms that coincide with these fixed norms on each $W_i$.   
$$R = \{\norm{\cdot} \in \mathcal{N}(V,K)^{\operatorname{diag}} \ | \ \exists (d_1,d_2, \ldots, d_l) \in \R^l {\text{ s.t. }} \norm{w} = e^{d_i}\norm{w}_i \ \forall w \in W_i \} / \sim$$
This is clearly a convex $\rho(\Delta)$ invariant subset of the building isometric to a Euclidean space of dimension $l-1$. Since $\rho(\Delta)$ fixes $l$ points on the boundary of this Euclidean space it must fix the whole boundary. Hence it acts on $R$ by translations. This give a homomorphism $\rho(\Delta) \arrow \R^{l-1}$ which must be trivial as $\rho(\Delta)$ is locally finite. Thus $\rho(\Delta)$ fixes $R$ pointwise and in particular it fixes many points in $X$ in contradiction to our definition of the measure $\mu_H$. 

This last argument is illustrated in figure \ref{fig:apt}, inside one apartment, of the simplicial building associated with $\PGL_4(k)$ over a local field. The basis corresponding to this apartment $\{e_1,e_2,e_3,e_4\}$ was chosen to respect the splitting $W_1 = \langle e_1,e_2 \rangle, W_2 = \langle e_3,e_4 \rangle$ as well as the individual fixed norms, which are here just the standard norms on these two subspaces. As is customary in the simplicial setting the homothety classes of norms are represented by homothety classes of unit balls with respect to these norms. The vertices thus correspond to lattices over the integer rings.

\begin{figure}[ht] 
\centering \def\svgwidth{300pt}
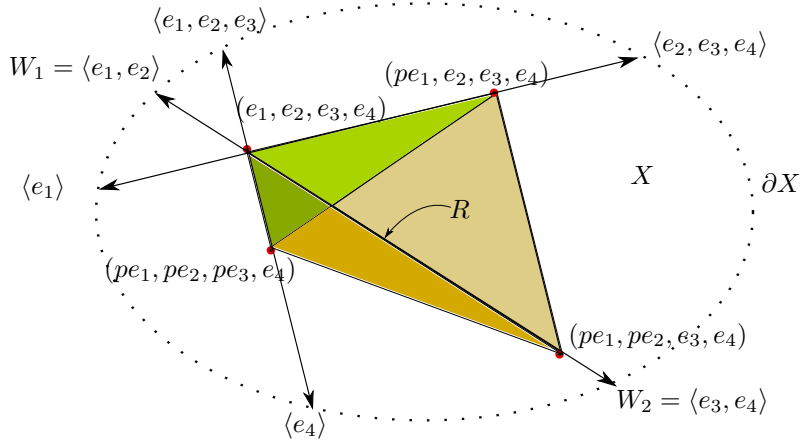 
\caption{An example of a $\rho(\Delta)$ fixed subspace $R$ within an apartment in the building of $\PGL_4(\Q_p)$.}
\label{fig:apt}
\end{figure}
\end{proof}

We will now aim at a contradiction, to the existence of $\mu_H$, by showing that $\mu_H$ almost every subgroup is totally reducible. As in Proposition \ref{prop:all_open}, the locally essential lemma \ref{lem:loc_ess} combined with Poincar\'{e} recurrence show that $\mu_H$ almost every subgroup is locally essential recurrent . 

It is a well known fact from ring theory that the a module is totally reducible if and only if it is generated by its irreducible components (see for example \cite[Theorem 1.4]{Weherfritz:linear_groups}). Thus the subspace $W = W(\Delta) < V$ generated by all the irreducible submodules of $\rho(\Delta)$ is a maximal $\rho(\Delta)$-totally reducible submodule. Assume toward contradiction that $W \ne V$ is not the whole space. We start by finding a finitely generated subgroup with the same maximal totally reducible subspace
\begin{lemma} \label{lem:fg_tot_red}
There is a finitely generated subgroup $\Sigma < \Delta$ such that $W$ is also a maximal totally reducible $\rho(\Sigma)$-module. 
\end{lemma}
\begin{proof}
Let  $W = \oplus_i W_i$ be the decomposition of $W_i$ into irreducible representations. Let $\Sigma_1 < \Sigma_2 < \ldots$ be finitely generated subgroups ascending to $\Delta$. Within each $W_i$ we can find some $W_{i,j}< W_i$ which is irreducible as a $\rho(\Sigma_j)$-module. Furthermore we can clearly arrange that $W_{i,j} \le W_{i,j+1}, \ \forall i, j$. By dimension considerations this process must stabilize. And it must stabilize with $W_{i,j} = W_{i}, \ \forall j \ge J$ since a module that is invariant under every $\Sigma_j$ must be $\Delta$-invariant too. Thus the maximal totally reducible subspace of $\Sigma_J$ contains $W$. 

Now let 
$U = W \oplus U_1 \oplus \ldots \oplus U_l$ be the maximal totally reducible subspace for $\rho(\Sigma_J)$. For each $U_i$ we can find some $\delta_i \in \Delta$ such that $\rho(\delta) U_i \cap W \ne \langle 0 \rangle$. Thus none of the spaces $U_i$ can be contained in the maximal totally reducible subspace for $\Sigma:= \langle \Sigma_J, \delta_1,\delta_2, \ldots, \delta_s \rangle$ and the maximal totally reducible subspace for this last group must be $W$ itself. 
\end{proof}

$\rho(\Gamma)$ contains a very proximal element $g$. Since $\rho(\Gamma)$ is strongly irreducible we can assume, after possibly replacing $g$ by a conjugate, that $W \not < \overline{H}_g$ and $\overline{v}_g \not \in W$. This immediately implies that 
$$\lim_{n \arrow \infty} d(\rho(g^n) W , \overline{v}_g) = 0.$$
In particular there exists some $N \in \N$ such that for every $n > N$ we have $W \lneqq \langle W, \rho(g^n)W \rangle$. Now consider the measure preserving dynamics on $\Sub(\Gamma)$. Let $\Sigma < \Delta$ is the finitely generated subgroup provided by Lemma \ref{lem:fg_tot_red}. Since $\Delta$ was chosen to be locally essential we know that $\mu(\Env(\Sigma)) > 0$. Now by the recurrence assumption the set of return times $N(\Delta, \Env(\Sigma), g) = \{n \in \N \ | \ g^n \Delta g^{-n} \in \Env(\Sigma)\}$ is infinite. And for every such such return time $n$ we have 
$$\sigma \rho(g^n) W = \rho(g^n) \rho(g^{-n} \sigma g^n) W = \rho(g^n) W \qquad \forall \sigma \in \Sigma.$$ 
Moreover since conjugation by $g$ induces an isomorphism between $W$ and $\rho(g^n)W$ as $\Sigma$-modules the latter is again totally reducible. Thus if we choose $n > N$ which is also in $N(\Delta, \Env(\Sigma), g)$ we have a subspace $W' = \langle W , \rho(g)W \rangle$ that is $\Sigma$-invariant, totally reducible as a $\rho(\Sigma)$ module and strictly larger than $W$; in contradiction to our construction of $\Sigma$. This contradiction finishes the proof.

\subsection{A nearly Zariski dense essential cover} \label{sec:Z_dense}
Let  $H = \overline{\rho(\Gamma)}^{Z}(K)$ be the $K$ points of the Zariski closure of $\rho(\Gamma)$ and $H^{(0)}$ the connected component of the identity in $H$. Our goal in this section is to find an essential cover by finitely generated subgroups $\Sigma$ such that $\overline{\rho(\Sigma)}^{Z}$ contains $H^{(0)}$. For this fix $\mu \in \IRS^{\chiNF}$. As a first step we claim that $\overline{\rho(\Delta)}^{Z} > H^{(0)}$ almost surely (compare \cite[Theorem 2.6]{7_sam}). 

Let $\hc = \Lie(\H)(K)$ be the Lie algebra and $\Gr(\hc)$ the Grassmannian over $\hc$. By taking the Lie algebra of the random subgroup we obtain a random point in this Grassmannian
\begin{eqnarray*}
\Psi: \Sub(\Gamma) & \arrow & \Gr(\hc) \\
\Delta & \mapsto & \left[ \Lie \left(\overline{\rho(\Delta)}^{Z} \right) \right]
\end{eqnarray*}
The push forward of the IRS gives rise to a $\rho(\Gamma)$-invariant measure on $\Gr(\hc)$. Since $\rho(\Gamma)$ is Zariski dense and $\H^0$ is simple, it follows from Furstenberg's lemma \cite[Lemma 2]{Furst:Borel_dense} that such a measure can be supported only on the $\Gamma^{(0)}$-invariant points $\{\hc, \{0\}\}$. By the previous subsection we already know that $\rho(\Delta)$ is almost surely infinite so that in fact $\psi(\Delta) = \hc$ almost surely. This means $\dim(\overline{\rho(\Delta)}^{Z}) = \dim(H)$. Hence $\overline{\rho(\Delta)}^{Z}$ itself has to be a subgroup of maximal dimension and it contains $H^{(0)}$ - even in positive characteristic. 

By the results of Sections \ref{sec:essinf}, \ref{sec:countable_p} $\rho(\Delta)$ is not locally finite almost surely. Since $\rho(\Delta)$ has a simple Zariski closure this immediately implies that in fact it is not even amenable. Thus by Lemma \ref{lem:fg_subgroup}  $\mu$-almost every $\Delta \in \Sub(\Gamma)$ contains a finitely generated non-amenable subgroup $\Sigma < \Delta$ with $\left(\overline{\Sigma}^{Z} \right)^{(0)} > H^{(0)}$. As in the proof of Lemma \ref{lem:loc_ess} we can ignore these subgroups $\Sigma$ which are not $\mu$-essential and obtain the desired essential cover. 

\subsection{Abundance of very proximal elements} \label{sec:essprox}
The goal of this section is to refine the essential cover we already have, to obtain a cover by groups which satisfy some desirable dynamical properties in their action on $\PP(K^n)$. We focus on two properties: strong irreducibility of $\rho(\Sigma)$ on $\PP(K^n)$ and the existence of an element $a_{\Sigma} \in \Sigma$ with $\rho(a_{\Sigma})$ very proximal. Where $\Sigma$ ranges over the essential group participating in the cover. The first property is automatic from what we have said so far as strong irreducibility of $\rho: \Sigma \arrow \PGL_n(K)$ is equivalent to the irreducibility of the Zariski connected component $\left(\overline{\rho(\Sigma)}^Z \right)^{(0)}$. We will refine the essential cover so that each of the covering subgroups will contain a very proximal element. 

So far we have a cover by all finitely generated essential subgroups $\Sigma$ that satisfy the following conditions (i) non-amenable and (ii) satisfy $\overline{\Sigma}^Z > H^{(0)}$ (iii) $\rho(\Sigma)$ is strongly irreducible. Since both these conditions are monotone Lemma \ref{lem:refine} implies that it is enough to find, for every such essential subgroup $\Sigma$, a larger essential subgroup $\Sigma < \Theta$ containing such a very proximal element. 

By Lemma \ref{lem:lip->very-proximal}, in order to establish the existence of a very proximal element, it is enough to find an element with a small enough Lipschitz constant on an arbitrarily small open set. Let $(\epsilon',r',c') = \left( \epsilon (\Sigma),r(\Sigma),c(\Sigma) \right)$ be the constants, governing the connection between the proximality and Lipschitz constants, provided by that Lemma. Recall also that we have fixed an $(r, \epsilon)$-very proximal element $g \in \Pi$ together with our linear representation. Since $\Sigma$ is strongly irreducible, there exists an element $x \in \Sigma$ such that $x$ moves the attracting point $\overline{v}_{g^{-1}}$ of $g^{-1}$ outside the repelling hyperplane $\overline{H}_g$ of $g$. Set $d=d(x \overline{v}_{g^{-1}},\overline{H}_g)$ to be the projective distance between the two. Now consider the element
$$
 y= g^{m} x g^{-m},
$$
We claim that if $m > m_0$ for some $m_0$ then $y$ is $\epsilon'$-Lipschitz on a small neighborhood of $\overline{v}_{g^{-1}}$. It follows from Lemma \ref{fix} that for a large enough $m$ the element $\rho (g^{-m})$ is $\eta$-very proximal with the same attracting point and repelling hyperplane as $\rho (g^{-1})$; where $\eta >0$ can be chosen arbitrarily small. By Lemma \ref{lem:contracting=Lipschitz} $\rho (g^{-m})$ is $C_1 \eta^2$-Lipschitz on some small open neighborhood $O$, away from $\overline{H}_{g^{-1}}$, for some constant $C_1$ depending only on $\rho(g)$. If we take $O$ to be small enough and $m$ large enough then $d \left( \rho (x g^{-m})(O),\overline{H}_g \right) > d/2$. 
By Lemma \ref{lem:contracting=Lipschitz} again, the element $\rho (g^{m_j})$ is
$C_2\epsilon^2$-Lipschitz on $\rho (x g^{-m})(O)$. Now since $\rho (x)$ is bi-Lipschitz with some constant $C_3$ depending on $\rho (x)$, it follows that $\rho (y) = \rho (g^{m} x g^{-m})$ is $C_4 \epsilon^4$-Lipschitz on $O$, where $C_4$ depends only on $\rho (g)$
and $\rho (x)$. If we require also $\epsilon' \leq C_4^{-1/3}$ then $\rho (y)$ is $\epsilon'$-Lipschitz on $O$ as we claimed. The dynamics of the action of the element $\rho(g^m x g^{-m})$ on $\PP(K^n)$ is depicted on the left side of Figure \ref{fig:contracting}. 

\begin{figure}[ht] 
\centering \def\svgwidth{300pt}
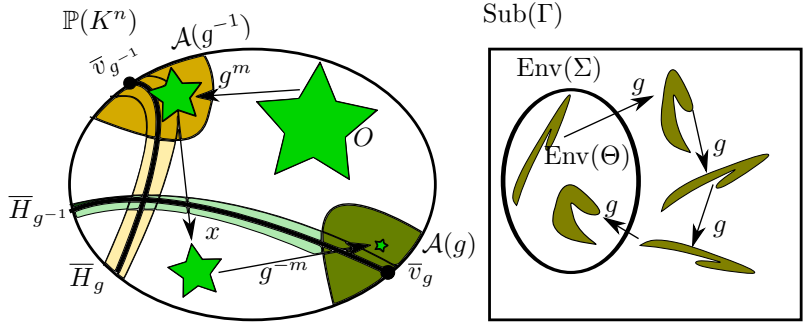 
\caption{The dynamics of $\rho(g^m x g^{-m})$ on the $\PP(K^n)$ and $\Sub(\Gamma)$.}
\label{fig:contracting}
\end{figure}

Now we analyze the action of the same element $g^mxg^{-m}$ on $\Sub(\Gamma)$. Poincar\'{e} recurrence yields infinitely many values of  $m$ for which 
$$\mu \left( \Env(\Sigma) \cap \Env (g^m \Sigma g^{-m}) \right) = \mu \left( \Env( \langle \Sigma, g^m \Sigma g^{-m} \rangle) \right) > 0.$$ We choose $m \in \N$ for which this is true and at the same time the reasoning of the previous paragraph applies so that $y = g^m x g^{-m}$ is $\epsilon'$-Lipschitz contained in the essential subgroup $\Theta := \langle \Sigma,  g^{m} \Sigma g^{-m} \rangle$. Since by assumption $\rho \left(\Sigma \right)$ is strongly irreducible Lemma \ref{lem:lip->very-proximal} provides elements $f_1,f_2 \in \Sigma$ such that $a_{\Theta} := y f_1 y^{-1} f_2 \in \Theta$ is $(r',c' \sqrt{\epsilon'})$-very proximal. Which completes the proof. Let $\Ac(a_{\Theta}^{\pm}), \Rc(a_{\Theta}^{\pm})$ be attracting and repelling neighborhoods for the elements we have just constructed.
%
%
By our construction all the groups $\Theta$ are again finitely generated, as they are generated by two finitely generated subgroups. Hence if 
$$\Fc := \left\{  \Theta \in \Eg \ \left | 
\begin{array}{l} 
\rho(\Theta)  {\text{ is strongly irreducible}} \\
\exists a \in \Theta {\text{ with }} \rho(a) {\text{ very proximal and }} \rho(a) \in \left(\overline{\rho(\Theta)}^Z \right)^{(0)} \\
\Theta {\text{ is finitely generated}}
\end{array} \right. \right \}$$
We have just shown that $\Fc$ covers $\IRS^{\chiNF}(\Gamma)$. 

\subsection{Putting the very proximal elements $\{a_{\Sigma}\}$ in general position}
In the construction of the very proximal elements $\{a_{\Sigma} \ | \ \Sigma \in \Fg\}$ in Subsection \ref{sec:essprox} no connection was established between these different elements. As a first step towards arranging them we would like, each one of them separately,  to form a ping-pong pair with our fixed very proximal element $g \in \Gamma$. Moreover we would like this to be realized with repelling and attracting neighborhoods:
$$\Rc(g^{\pm 1}), \Ac(g^{\pm 1}), \Rc(a_{\Sigma}^{\pm 1}), \Ac(a_{\Sigma}^{\pm 1}),$$
such that $\Rc(g^{\pm 1}), \Ac(g^{\pm 1})$ are independent of $\Sigma$. This will be done, at the expense of refining the cover again $\Fgg <_{\IRS^{\chiNF}(\Gamma)} \Fc$. 

Let $h \in \Pi$ be any very proximal element satisfying the ping-pong lemma conditions with $g$. Assume that $\Rc(g^{\pm 1}), \Ac(g^{\pm 1}), \Rc(h^{\pm 1}), \Ac(h^{\pm 1})$ are the corresponding attracting and repelling neighborhoods. For the sake of this proof only we will call an element $h'$ {\it{good}} if it is very proximal with attracting and repelling neighborhoods contained in these of $h$ as follows: $\Ac(h'^{\pm 1}) \subset \Ac(h),\Rc(h'^{\pm 1}) \subset \Rc(h^{-1})$. Such a {\it{good}} element is indeed good for us: the pair $\{h',g \}$ satisfies the conditions of the ping-pong lemma, with the given attracting and repelling neighborhood. 

Let $\Fgg \subset \Fc$ be the collection of essential subgroups in $\Fc$ that contain a good element. Since $\Fc$ is monotone Lemma \ref{lem:refine} implies that it is enough to show that any given $\Sigma \in \Fc$ is contained in a larger subgroup $\Theta \in \Fgg$. So fix $\Sigma \in \Fc$ and let $a = a_{\Sigma}$ be the associated very proximal element with attracting and repelling neighborhoods $\Ac(a_{\Sigma}^{\pm 1}), \Rc(a_{\Sigma}^{\pm 1})$. We can always assume that $a,h$ are in general position with respect to each other, that is to say that some high enough powers of them play ping-pong. Indeed conjugation just moves the attracting points and repelling hyperplanes around $\overline{v}_{\sigma a^{\pm} \sigma^{-1}} = \rho(\sigma) \overline{v}_{a^{\pm}}, \overline{H}_{\sigma a^{\pm} \sigma^{-1}} = \rho(\sigma) \overline{H}_{a^{\pm}}$, and since by assumption $\Sigma$  is strongly irreducible we can choose $\sigma \in \Sigma$ that will bring $a$ to a general position with respect to $h$. 

Now as $a,h$ are in general position we claim that $h^l a h^{-l}$ is good for any sufficiently large $l$. This is statement is clear upon observation that
\begin{eqnarray*}
\Ac(h^l a^{\pm} h^{-l})  & = & \rho(h^l) \Ac(a^{\pm}) \\
\Rc(h^l a^{\pm} h^{-l})  & = & \rho(h^l) \Rc(a^{\pm})
\end{eqnarray*}
Consider for example the second equation. Since $a$ and $h$ are in general position we know that $d \left( \overline{v}_{h^{-1}},  \Rc(a) \cup \Rc(a^{-1}) \right) > 0$. Since $h$ satisfies the conditions of Lemma \ref{fix} we can choose $l$ large enough so that $\Ac(h^{-l}) \subset (\overline{v}_{h^{-1}})_{d/2}$ and hence $\Ac(h^{-l}) \cap \left(\Rc(a) \cup \Rc(a^{-1}) \right)  = \emptyset$. Now the equation above implies that 
$$\Rc(h^l a^{\pm} h^{-l})  = \rho(h^l) \Rc(a^{\pm}) \subset \rho(h^l) \left(\PP(k^n) \setminus \Ac(h^{-l}) \right) \subset \Rc(h^l) \subset \Rc(h).$$
Where the inclusion before last follows from the fact that $\rho(h^{-l}) \left(\PP(k^n) \setminus \Rc(h^l) \right)  \subset \Ac(h^{-l})$. The second case follow in a similar fashion. 

On the space $\Sub(\Gamma)$ we repeat the argument from the previous section. By Poincar\'{e} recurrence the group $\Theta:=\langle \Sigma, h^{l} \Sigma h^{-l} \rangle$ is again essential for infinitely many values of $l$. We choose one such value to make $\Theta$ essential and at the same time large enough so that  $a_{\Theta} := h^l a h^{-l}$ is good. This completes the proof that $\Fgg$ covers $\IRS^{\chiNF}(\Gamma)$, where $\Fgg$ is the set
$$\left\{  \Theta \in \Fg \ \left | 
\begin{array}{l} 
\rho(\Theta)  {\text{ is strongly irreducible}} \\
\exists a \in \Theta {\text{ with }} \rho(a) {\text{ good, very proximal and }} \rho(a) \in \left(\overline{\rho(\Theta)}^Z \right)^{(0)} \\
\Theta {\text{ is finitely generated}}
\end{array} \right. \right \}.$$
The following proposition gives a good idea of the possible use we can make of such a cover in the construction of free groups: 
\begin{proposition} \label{prop:independent}
Assume that $g$ and $\{a_{\Theta} \ | \ \Theta \in \Fgg \}$ are very proximal elements as constructed above. Namely they admit attracting and repelling neighborhoods $\Ac(g^{\pm})$, $\Rc(g^{\pm})$, $\Ac(a_{\Theta}^{\pm})$, $\Rc(a_{\Theta}^{\pm})$ such that for every $\Theta \in \Fgg$ the elements $\{a_{\Theta}, g\}$ satisfy the conditions of the ping-pong lemma \ref{lem:ping-pong} with respect to the given neighborhoods. 

Then for every map $\psi: \Z \arrow \Fgg$ the elements 
$$\left \{g^{i} a_{\psi(i)} g^{-i} \ | \ i \in \N \right \},$$
satisfy the conditions of the ping-pong lemma and are therefore independent. 
\end{proposition}
\begin{proof}
Consider the closed sets
$$\hat{D}:=\PP(k^n) \setminus \left(\Ac(g) \cup \Ac(g^{-1}) \right), \qquad D :=\PP(k^n) \setminus \left(\Ac(g) \cup \Ac(g^{-1}) \cup \Rc(g) \cup \Rc(g^{-1})\right),$$
These satisfy the following properties:
\begin{enumerate}
\item \label{itm:all_in_hat}
$\Ac(a_{\Theta}^{\pm}) \cup \Rc(a_{\Theta}^{\pm}) \subset \hat{D} \ \ \forall \Theta \in \Fgg$.
\item \label{itm:some_in_D} $\Ac(a_{\Theta}^{\pm}) \subset D \ \ \forall \Theta \in \Fgg$.
\item \label{itm:disjoint_translates} $\rho(g^i)(\hat{D}) \cap \rho(g^j)(D) = \emptyset, \ \ \forall i \ne j \in \Z$ 
\end{enumerate}
Indeed the first two properties follow directly from our ping-pong assumption on the pairs $\{g,a_{\Theta}\}$. As for the third property, after applying an appropriate power of $\rho(g)$ and possibly changing $i$ and $j$, we may assume without loss of generality that $i =0$ and that $j >0$; but $\rho(g^j)(D) \subset \Ac(g)$ and $\hat{D} \cap \Ac(g) = \emptyset$.  

Now these three properties together imply the statement. Indeed given $i \ne j \in \Z$ we have to show that 
$$\Ac(g^{i} a_{\psi(i)}^{\pm} g^{-i}) \cap \left[\Ac(g^{j} a_{\psi(j)}^{\pm} g^{-j}) \cup  \Rc(g^{j} a_{\psi(j)}^{\pm} g^{-j}) \right] = \emptyset$$
But $\Ac(g^{i} a_{\psi(i)}^{\pm} g^{-i}) = \rho(g^{i}) \left(\Ac(a_{\psi(i)}^{\pm}) \right) \subset \rho(g^{i})(D)$ while for a similar reason we have $\Ac(g^{j} a_{\psi(j)}^{\pm} g^{-j}) \cup  \Rc(g^{j} a_{\psi(j)}^{\pm} g^{-j}) =   \rho(g^{j}) \left(\Ac(a_{\psi(j)}^{\pm}) \cup \Rc(a_{\psi(j)}^{\pm}) \right) \subset \rho(g^{j})(\hat{D})$ which concludes the proof by Property (\ref{itm:disjoint_translates}) above. A case that is particularly easy to visualize is the rank one situation where the very proximal elements exhibit a north-south type dynamics on the projective plane. The situation in this specific case is depicted in figure \ref{fig:arrange} below. 
\end{proof}

\begin{figure}[ht] 
\centering \def\svgwidth{200pt}
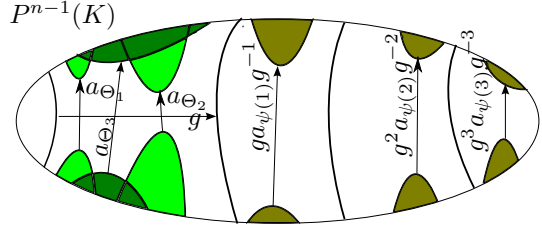 
\caption{Arranging very-proximal elements into a ping-pong setting}
\label{fig:arrange}
\end{figure}

\subsection{Poincar\'{e} recurrence estimates for multiple measures} \label{sec:Poincare}
Before proceeding to the proof of Theorem \ref{thm:main_tec}. We will need the following quantitative version of Poincar\'{e} recurrence.  
\begin{lemma} \label{lem:poincare} (Poincar\'{e} recurrence)
Let $(X,\Bc,\mu,T)$ be a probability measure preserving action of $\Z = \langle T \rangle$, $A \in \Bc$ a set of positive measure and $\epsilon >0$. Then there exists a number $n = n(A,\mu ,\epsilon) \in \N$, such that for every $N \in \N$ we have 
$$
\mu \left( A \setminus \left( \bigcup_{i=N}^{N+n-1}T^{i}A \right ) \right) < \epsilon.
$$
\end{lemma}
\begin{proof}
Let $V_m = \left \{x \in A \ | \ T^m x \in A, {\text{ but }} T^{l}x \not \in A  \ \ \forall 1 \le l < m \right \}$. Clearly these sets are measurable and disjoint. Poincar\'{e} recurrence implies that their union covers all of $A$ up to $\mu$-nullsets. This enables us to cover the forward $\langle T \rangle$-orbit of $\Env(\Sigma)$ by a so called  Kakutani-Rokhlin tower (see figure \ref{fig:KM}):
\begin{eqnarray*}
\cup_{n \in \N} T^{n}A & = & \left[ V_1 \right] \\ 
& \sqcup & \left[V_2 \sqcup T^{1}V_2  \right] \\
& \sqcup & \left[V_3 \sqcup T^{1}V_3 \sqcup T^{2}V_3  \right] \sqcup \ldots 
\end{eqnarray*} 
Where all the sets are disjoint and equality is up to nullsets. Denoting the tail of this tower by 
$$
\Tail(m) =   \left[V_{m+1} \sqcup \ldots \sqcup T^{m}V_{m+1} \right] \sqcup  \left[V_{m+2} \sqcup \ldots \sqcup T^{m+1} V_{m+2} \right] \sqcup \ldots,
$$
we choose $n = n(T,\mu,\epsilon)$ to be the minimal number such that $\mu(\Tail(n))< \epsilon$. The lemma now follows from the obvious inclusion
$$A \setminus \left( \bigcup_{i=N}^{N+n-1}T^{i}A \right )  \subset T^{-N}(\Tail(n)).$$  
\begin{figure}[t] 
\centering \def\svgwidth{100pt}
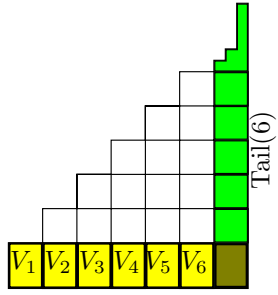 
\caption{A Kakutani-Rokhlin tower} 
\end{figure} \label{fig:KM}
\end{proof}

\subsection{An independent set} 
We proceed to our last refinement: a collection $\Hg \subset \Fgg$ of finitely generated essential subgroups together together with a set of associated very proximal elements $B := \{b_{\Theta} \in \Theta  \ | \ \Theta \in \Hg \}$ which satisfy the following two conditions:
\begin{description}
\item[Ind] \label{itm:Ind} $B = \{b_{\Theta} \ | \ \Theta \in \Hg \}$ forms an independent set.  
\item[Cov] \label{itm:Cov} $\Hg$ covers $\IRS^{\chiNF}$. 
\end{description}
To achieve this, we can no longer appeal to Lemma \ref{lem:refine} as our new requirement {\bf{Ind}} can no longer be verified one group at a time. To obtain an independent set we apply Proposition \ref{prop:independent} with a function $\psi: \Z \arrow \Fgg$ satisfying the following property:
\begin{description}
\item[{\bf{Rpt}}] For every $(n, \Sigma) \in \N \times \Fgg$ there is some index $N = N(n,\Sigma)$ such that $\psi(i) = \psi(i+1) = \psi(i+2) = \ldots = \psi(i+n-1) = \Sigma$.
\end{description}
to obtain a sequence of subgroups $\Theta_j = \langle \Sigma, g^{j} \Sigma g^{-j} \rangle$ and an associated sequence of elements 
$$\left\{ b_{\Theta_j} = g^{j} a_{\psi(j)} g^{-j} \in  \Theta_j \qquad j \in \N \right\}.$$ Proposition \ref{prop:independent} guarantees that the elements $\{b_{\Theta_j} \ | \ j \in \N\}$ form a ping-pong tuple and are hence independent. It is not necessarily true that every $\Theta_j$ is essential but to construct $\Hg$ and $B$ we take only the indexes $j$ for which this is the case:\begin{eqnarray*}
\Hg & = & \{ \Theta_j \ | \ j \in \Z, \Theta_j {\text{ is }} \mu-{\text{essential for some }} \mu \in \IRS^{\chiNF} \} \\
B & = & \{b_{\Theta} \ | \ \Theta \in \Hg \}.
\end{eqnarray*}
Clearly $\Hg \subset \Fgg$. Indeed each $\Theta \in \Hg$ is finitely generated as it is generated by two finitely generated subgroups; it is strongly irreducible as it contains a strongly irreducible subgroup and it contains the very proximal element $b_{\Theta}$. The independence condition ({\bf{Ind}}) for the elements of $B$ follows directly from Proposition \ref{prop:independent}. 

As for the covering condition ({\bf{Cov}}), this is where our uniform estimates on Poincare\'{e} recurrence come into play. Fix $\Sigma \in \Fgg$, $\epsilon > 0$ and let $\mu \in \IRS^{\chiNF}(\Gamma)$ be a measure according to which $\Sigma$ is $\mu$-essential. Usually there will be uncountably many such measures since $\Hg$ is merely countable while $\IRS^{\chiNF}(\Gamma)$ is not! Still we can apply Lemma \ref{lem:poincare} to the dynamical system $(X = \Sub(\Gamma), \Bc,\mu,g)$ and to the positive measure set $A = \Env(\Sigma)$. Let $n = n(A,\mu,\epsilon) = n(\Sigma,\mu,\epsilon)$ be the constant given by that lemma so that for every $N \in \Z$ we have
$$
\mu \left\{\Env(\Sigma) \setminus \bigcup_{j = N}^{N + n-1} \Env(\langle \Sigma, g^{j} \Sigma g^{-j} \rangle ) \right\}  < \epsilon.$$ 
By our construction of the function $\psi$ there exists an index $N = N(n,\Sigma)$ such that 
$\Theta_j = \langle \Sigma, g^{j} \Sigma g^{-j} \rangle \in \Hg$ for every $N(n,\Sigma) \le j \le N(n,\Sigma)+n-1$. Or at least for every such $j$ such that $\Theta_j$ is essential. Since anyhow these that are non essential have envelopes of measure zero we have
$$\mu \left\{\Env(\Sigma) \setminus \bigcup_{\Theta \in \Hg} \Env(\Theta) \right\}  \le \mu \left\{\Env(\Sigma) \setminus \bigcup_{j = N}^{N+n-1} \Env(\Theta_j) \right\} < \epsilon.$$
Since $\mu, \epsilon, \Sigma$ are arbitrary we deduce that $\Hg$ covers $\IRS^{\chiNF}$ as promised.  

\subsection{Conclusion}
Here we complete the proof of Theorem \ref{thm:main_tec} thereby completing the proof of the main theorem. Let $\Hg = \{\Theta_1, \Theta_2, \ldots \}, B = \{b_1,b_2, \ldots \}$ be the groups and very proximal elements constructed in the previous section - with a new enumeration making the indices consecutive. No further refinement of $\Hg$ the collection of essential subgroups will be needed; what we do want is a much richer free subgroup containing representatives of all double cosets of the form $\Theta_q \gamma \Theta_p$. For this fix a bijection 
\begin{eqnarray*}
\N & \arrow & \N^2 \times \Gamma \\
k & \mapsto & \left( p(k),q(k),\gamma(k) \right)
\end{eqnarray*}

We proceed to prove Theorem \ref{thm:main_tec} by induction on $k$.  Assume that after $k$ steps of induction we have constructed group elements $\{f_i \in \Gamma \ | \ 1 \le i \le k \}$ and a sequence of numbers $\{n_p(k) \in \N \ | \ p \in \N \}$ with the following properties: 
\begin{itemize}
\item $\rho(f_i)$ is very proximal for each $1 \le i \le k$. 
\item $n_p(k) \ge n_p(k-1), \ \forall p \in \N$. 
\item $f_i \in \Theta_{q(i)} \gamma(i) \Theta_{p(i)}, \ \forall 1 \le i \le k$. 
\item The elements 
$$B(k) := \{f_1, f_2 \ldots ,f_{k}, b_1^{n_1(k)},b_2^{n_2(k)} ,\ldots \},$$
form an (infinite) ping-pong tuple - satisfying the conditions of Lemma \ref{lem:ping-pong}. 
\end{itemize} 

For the basis of the induction we take $k=0, n_p(0)=1$ and $B(0) = B$. Now assume that we already have all of the above for $k-1$ and let us show how to carry out the induction. Set $p = p(k), q = q(k), \gamma = \gamma(k)$. Since $\Theta_{p}, \Theta_{q} \in \Hg \subset \Fgg \subset \Fg$ both these groups act strongly irreducibly on the projective space. So there exist elements $x= x(k) \in \Theta_{p}, y = y(k) \in \Theta_{q}$ such that:
\begin{eqnarray}
\label{eq1} \rho (y \gamma x) \overline{v}_{b_{p}}  
      	& \not \in & \overline{H}_{b_{q}} \\
\nonumber \rho (y \gamma x)^{-1} \overline{v}_{b_{q}^{-1}}  
      	& \not \in & \overline{H}_{b_{p}^{-1}} \\
\nonumber \rho(y \gamma x)^{-1} \overline{v}_{b_q} 
	& \not \in & \overline{H}_{b_p} \\
\nonumber \rho(y \gamma x) \overline{v}_{b_p^{-1}} 
	& \not \in & \overline{H}_{b_q^{-1}} 
\end{eqnarray}
Indeed the set of elements in $\left(\Theta_p \cap \G^{(0)} \right) \times \left(\Theta_q \cap \G^{(0)} \right)$ that satisfy each one of the above conditions is Zariski open and non-trivial by the strong irreducibility. Hence their intersection is non-empty. 

We define 
$$f_k := b_{q}^{l_2} y \gamma x b_{p}^{l_1}, $$ 
Denote by $(\cdot)_{\epsilon}$ an $\epsilon$-neighborhood in $\PP(k^n)$. We claim that for a small given $\epsilon > 0$ if $l_1,l_2$ are large enough then there exists some $0 < \epsilon' < \epsilon$ such that 
\begin{eqnarray} \label{eqn:inclusions}
\nonumber \rho(f_k) \left(\PP^{n-1}(k) \setminus \left[ \left(\overline{H}_{b_p}\right)_{\epsilon} \setminus \left(\overline{H}_{b_p}\right)_{\epsilon'} \right] \right)    
   & \subset &  \left( \overline{v}_{b_q} \right)_{\epsilon'} \setminus\left( \overline{v}_{b_q} \right)_{\epsilon'}\\
\rho(f_k^{-1}) \left(\PP^{n-1}(k) \setminus \left[ \left(\overline{H}_{b_q^{-1}}\right)_{\epsilon}   \setminus \left(\overline{H}_{b_q^{-1}}\right)_{\epsilon'}  \right] \right)
   & \subset &  \left( \overline{v}_{b_q^{-1}} \right)_{\epsilon} \setminus\left( \overline{v}_{b_q^{-1}} \right)_{\epsilon'}. 
\end{eqnarray}
 Indeed we can write
\begin{eqnarray*}
\rho(f_k) \left(\PP^{n-1}(k) \setminus \left(\overline{H}_{b_p}\right)_{\epsilon}  \right) & \subset &  \rho(b_{q}^{l_2} y \gamma x) \left( \left(\overline{v}_{b_p}  \right)_{\epsilon_1} \right) \\
    & \subset &   \rho(b_{q}^{l_2}) \left( \left( \rho( y \gamma x) \overline{v}_{b_p}  \right)_{\epsilon_2}  \right) \\
    & \subset & \rho(b_{q}^{l_2}) \left( \left(\PP^{n-1} \setminus \overline{H}_{b_q}  \right)_{\epsilon_2} \right) \\
    & \subset & \left( \overline{v}_{b_q} \right)_{\epsilon} \setminus\left( \overline{v}_{b_q} \right)_{\epsilon'} 
\end{eqnarray*}
Where, since $b_{p}$ is very proximal, the first inclusion holds for arbitrarily small values of $\epsilon_1$, whenever  $l_1$ is large enough. Consequently $\epsilon_2$ can be set arbitrarily small by continuity of $\rho(y \gamma x)$ on the projective plane. In view of the first line in Equation (\ref{eq1}) we can choose $\epsilon_2 < \frac{1}{2}d\left(\rho(y \gamma x)\overline{v}_{b_p}, \overline{H}_{b_q}\right)$ which accounts the third line. Now by choosing $l_2$ to be large enough, $b_q^{l_2}$ becomes very proximal with attracting and repelling neighborhoods arbitrarily close to $\overline{v}_{b_q}$ and $\overline{H}_{b_q}$ respectively. In particular we can arrange for the image to be contained in $(\overline{v}_{b_p})_{\epsilon}$ as indicated. The fact that we can exclude from the image a ball of some positive radius $\epsilon' > 0$ comes from the fact that $\overline{v}_{b_q}$ is a fixed point for the homeomorphism $\rho(b_q^{l_2})$ and that, by the above observations, we may assume that $\epsilon_2 < \frac{1}{2} d(\overline{H}_{b_q}, \overline{v}_{b_q})$. 

\begin{figure}[ht] 
\begin{center}
\includegraphics[width=8cm]{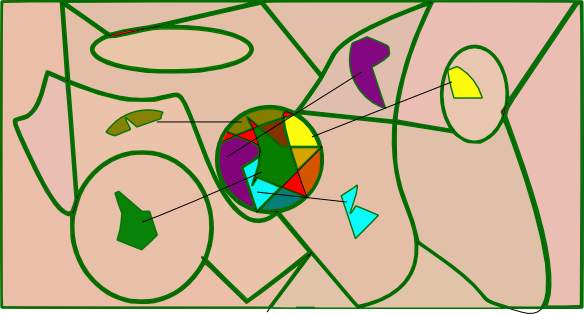}
\end{center}
\caption{Final refinement} 
\label{fig:final}
\end{figure}

And on the other hand 
\begin{eqnarray*}
\rho(f_k) \left( \left(\overline{H}_{b_p}\right)_{\epsilon'}  \right) & \subset &  
   \rho(b_{q}^{l_2} y \gamma x) \left( \left(\overline{H}_{b_p}\right)_{\epsilon_1} \right) \\
    & \subset &   \rho(b_{q}^{l_2}) \left( \left( \rho( y \gamma x) \overline{H}_{b_p}  \right)_{\epsilon_2}  \right) \\
    & \subset & \rho(b_{q}^{l_2}) \left(\PP^{n-1} \setminus \left( \overline{v}_{b_q}  \right)_{\epsilon_2} \right) \\
    & \subset & \left( \overline{v}_{b_q} \right)_{\epsilon} \setminus\left( \overline{v}_{b_q} \right)_{\epsilon'} 
\end{eqnarray*} 
Where the first equation can be satisfied, for arbitrarily small values of $\epsilon_1$, if we assume that $\epsilon'$ is small enough; by the fact that $\H_{b_p}$ invariant under the continuous map $b_p^{l_1}$ and that the latter map admits some global Lipschitz constant on the whole of $\PP(k^n)$.  The second inclusion follows, for arbitrarily small values of $\epsilon_2$, using the Lipschitz nature of $\rho(y \gamma x)$ combined with the ability to bound $\epsilon_1$. Now using the third line in Equation \ref{eq1} we can choose $\epsilon_2 < d(\overline{v}_{b_q}, \rho(y \gamma x) \overline{H}_{b_p})$. Finally the last line follows from the previous one and from the fact that $\overline{v}_{b_q}$ is $\rho(b_{q})$-invariant. 

These last two calculations combine to yield the first Equation in \ref{eqn:inclusions}. The second equation there follows in a completely symmetric fashion. 

Equations \ref{eqn:inclusions}, applied to a small enough value of $\epsilon$, imply the existence of $\epsilon' > 0$ and $l_1,l_2 \in \N$ such that $\rho(f_k)$ becomes a very proximal element with attracting and repelling neighborhoods as follows:
\begin{eqnarray*} \label{eq2}
\overline{\Rc(f_k)} & \subset & \Rc(b_p) \setminus \left(\overline{H}_{b_p}\right)_{\epsilon'}  \\
\overline{\Ac(f_k)} & \subset & \Ac(b_q) \setminus\left( \overline{v}_{b_q} \right)_{\epsilon'} \\
\overline{\Rc(f_k^{-1})} & \subset & \Rc(b_q^{-1}) \setminus \left(\overline{H}_{b_q^{-1}}\right)_{\epsilon'}  \\
\overline{\Ac(f_k^{-1})} & \subset & \Ac(b_p^{-1}) \setminus\left( \overline{v}_{b_p^{-1}} \right)_{\epsilon'}. 
\end{eqnarray*}
Since the elements $b_p, b_q$ are very proximal, and in addition satisfy the conditions of Lemma \ref{fix}, powers of them will have the same attracting points and repelling hyperplanes, and we can arrange for the attracting and repelling neighborhoods to be in arbitrarily small neighborhoods of these points and planes  respectively. Thus choose $n_p(k) > n_p(k-1), n_q(k) > n_q(k-1)$ that give rise to 
$\Rc \left(b_r^{\pm n_r(k)} \right) \subset \left(\overline{H}_{b_r^{\pm}} \right)_{\epsilon'}, \Ac \left(b_r^{\pm n_r(k)} \right) \subset \left( \overline{v}_{b_r^{\pm}} \right)_{\epsilon'}$, for every $r \in \{p,q\}$. For every other $r \in \N \setminus \{p,q\}$ we just set $n_r(k) = n_r(k-1)$. 

We leave it to the reader to verify that the induction holds. The only thing that really requires verification is the ping-pong conditions for:
$$\{f_1, f_2 \ldots ,f_{k}, b_1^{n_1(k)},b_2^{n_2(k)} ,\ldots, b_p^{n_p(k)}, 
\ldots, b_p^{n_p(k)}, \ldots,   \},$$
which follows quite directly based only on the induction hypothesis and the specific repelling and attracting neighborhoods we have found for $f_k^{\pm 1}, b_p^{\pm n_p(k)}, b_q^{\pm n_q(k)}$.
 
\appendix
\section{A relative version of Borel density theorem - By Tsachik Gelander and Yair Glasner} \label{appendix}
In this appendix we prove a relative version of the Borel Density Theorem for Invariant Random Subgroups in countable linear groups. 

\begin{theorem} (Relative version of the Borel density theorem)
Let $\Gamma < \GL_n(F)$ be a countable linear group with a simple, center free, Zariski closure. Then every non-trivial invariant random subgroup $\Delta \leftIRS \Gamma$ is Zariski dense almost surely.
\end{theorem}
\begin{proof}
We first claim that $\Delta$ is either trivial or infinite almost surely. Assume that $\Delta$ is finite with positive probability, in this case we may further assume that $\Delta$ is finite almost surely by conditioning on the event $\operatorname{Fin} := \{\Delta \ | \ |\Delta|< \infty\} \subset \Sub(\Gamma)$. Since the invariant measure is supported on the countable set $\Fin$, every ergodic component is a finite conjugacy class. Thus the IRS is supported on finite subgroups with finite index normalizers. Passing to the Zariski closure we conclude that $[\G:N_{\G}(\Delta)] < \infty$ and since the normalizer of a finite group is algebraic $N_{\G}(\Delta) > \G^{(0)}$ almost surely. Since $\G^{(0)}$ is center free $\Delta = \trivgp$. 

We apply Theorem \ref{prop:good_rep} to the group $\Gamma$; obtaining a local field $F < k$, a faithful, unbounded, strongly  irreducible projective representation $\rho: \G(k) \arrow \PGL_n(k)$ defined over $k$.   Let $\gc = \Lie(G)$ be the Lie algebra and $\Gr(\gc)$ the Grassmannian over $\gc$. By taking the Lie subalgebra corresponding to the Zariski closure of the random subgroup we obtain a random point in this Grassmannian
\begin{eqnarray*}
\Psi: \Sub(\Gamma) & \arrow & \Gr(\gc)(k) \\
\Delta & \mapsto & \Lie \left(\overline{\Delta}^{Z} \right)
\end{eqnarray*}
The push forward of the IRS gives rise to a $\rho(\Gamma)$-invariant measure on $\Gr(\gc)$. Since $\rho(\Gamma)$ is unbounded, it follows from Furstenberg's lemma \cite[Lemma 2]{Furst:Borel_dense}, applied in each dimension separately, that such a measure is supported on a disjoint union of linear projective subspaces. However $\Gamma$ itself is Zariski dense and $G$ is simple, so the only (ergodic) possibilities are the Dirac $\delta$ measures supported on the trivial and the whole Lie algebra: $\{\gc, \{0\}\}$. In the first case $\overline{\Delta}^{Z}$ is open and unbounded and hence equal to $\G$ (c.f. \cite{Shalom:inv_meas}). In the second case $\overline{\Delta}^{Z}$ is finite and hence trivial by the previous paragraph.
\end{proof}

As a consequence we obtain an alternative direct proof of Corollary 1.10, which bypasses the main theorem of the paper, which we restate for the convenience of the readers.
\begin{corollary} \label{cor:main_irs_restate}
Let $\Gamma < \GL_n(F)$ be a countable linear group. If $\Delta \leftIRS \Gamma$ is an $\IRS$ in $\Gamma$ that is almost surely amenable, then $\Delta < A(\Gamma)$ almost surely. 
\end{corollary}
\begin{proof} (of Corollary \ref{cor:main_irs_restate})
Let $\Gamma$ be a linear group and $\Delta \leftIRS \Gamma$ an IRS that is almost surely amenable. We wish to show that $\Delta < A(\Gamma)$ is contained in the amenable radical almost surely. At first we assume that $\Gamma$ is finitely generated. 

If $\Gamma < \GL_n(\Omega)$ is linear over an algebraically closed field, let $\G = \overline{\Gamma}^{Z}$ be its Zariski closure, $\H = \G/\Rad \left(\G^{(0)} \right)$ the semisimple quotient and $\H^{(0)} = \prod_{\ell = 1}^{L} \H_{\ell}$ the decomposition of $\H^{(0)}$ as a product of simple factors. Set $\hc_{\ell} = \Lie(\H_{\ell})$ and consider the combined map:
$$\G^{(0)} \longrightarrow \G^{(0)}/\Rad(\G^{(0)}) \stackrel{\Ad}{\longrightarrow} \prod_{\ell=1}^{L} \GL(\hc_{\ell})$$
Applying this map to $\Gamma^{(0)} := \Gamma \cap G^{(0)}$ we obtain $L$ different representations with simple center free Zariski closures. Let $\Delta^{(0)} \leftIRS \Gamma^{(0)}$ be the restriction of the IRS to the Zariski connected component. For every $1 \le \ell \le L$ we know by the previous paragraph that $\chi_{\ell}(\Delta^{(0)}) \leftIRS \chi_{\ell}(\Gamma^{(0)})$ is either trivial or Zariski dense. Assume by way of contradiction that the latter holds. This means that $\chi_{\ell}(\Delta^{(0)})$ does not admit any normal solvable subgroup, as its Zariski closure is simple. Since $\chi_{\ell}(\Delta^{(0)})$ is amenable, the Tits alternative \cite[Theorem 2]{Tits:alternative} says that $\chi_{\ell}(\Delta^{(0)})$ is locally finite. Now we use for the first time our assumption that $\chi_{\ell}(\Gamma^{(0)})$ is finitely generated: By \cite[Corollary 4.8]{Weherfritz:linear_groups} it follows that the locally finite subgroup $\chi_{\ell}(\Delta^{(0)})$ contains a further finite index subgroup consisting only of unipotent elements - contradicting the assumption that $\chi_{\ell}(\Delta^{(0)})$ has a simple Zariski closure. Thus we have concluded that $\chi_{\ell}(\Delta^{(0)}) = \trivgp$ almost surely for every $1 \le \ell \le L$. 

Let $A_{\ell} = \ker(\chi_{\ell} \cap \Gamma^{(0)})$. We have just proved that $\Delta^{(0)} < S := \cap_{\ell = 1}^{L} A_{\ell}$. Thus if $\Phi: \G \arrow \H = \G / \Rad(\G^{(0)})$ is the map to the semisimple quotient then $\Phi(\Delta) \leftIRS(\Phi(\Gamma))$ is an invariant random subgroup supported on finite subgroups. The same argument used earlier shows that $\Phi(\Delta^{(0)})$ is centralized by $\H^{(0)}$. Thus $\Delta$ is almost surely contained in the normal amenable subgroup $\Gamma \cap \Phi^{-1} (Z_{\H} ( \H^{(0)}))$ which completes the proof in the finitely generated case. 

Assume now that $\Gamma < \GL_n(F)$ is not finitely generated and  let $\Gamma_1 < \Gamma_2  < \ldots$ be a sequence of finitely generated subgroups ascending to the whole group. Applying restriction we obtain $\Delta_i := \Delta \cap \Gamma_i \leftIRS \Gamma_i$ an IRS in $\Gamma_i$ which is again, almost surely amenable. Since the group $\Gamma_i$ is finitely generated we may deduce from the above paragraphs that $\Delta_i$ is almost surely contained in $A(\Gamma_i)$. Since an essential element element in $\Gamma_i$ is contained in $\Delta_i$ with positive probability this means that every essential element in $\Gamma_i$ is contained in $A(\Gamma_i)$. Now essential elements in $\Gamma$ are eventually essential in $\Gamma_i$ so, denoting by $\Eg$ the collection of essential elements in $\Gamma$ we have 
$$\Eg \subset \liminf_n A(\Gamma_i) := \bigcup_n \bigcap_{i \ge n} A(\Gamma_i)$$ 
$\liminf A(\Gamma_i)$ is amenable as an ascending union of amenable groups. It is also normal since if $\gamma \in \Gamma_m$ then for every $o \ge \max \{m,n\}$ we have $\cap_{i \ge n} A(\Gamma_i) \subset A(\Gamma_o)$ so that $\gamma \left( \cap_{i \ge n} A(\Gamma_i) \right) \gamma^{-1} \subset \gamma A(\Gamma_o) \gamma^{-1} = A(\Gamma_o)$ and consequently 
$$\gamma \left(\bigcap_{i \ge n} A(\Gamma_i) \right) \gamma^{-1} \subset \bigcap_{i \ge \max\{m,n\}} A(\Gamma_i).$$  
Being an amenable normal subgroup $\liminf_n A(\Gamma_n)$ is contained in the amenable radical $A(\Gamma)$ which implies that every essential element is contained in $A(\Gamma)$. Lemma \ref{lem:loc_ess}, applied only to cyclic groups, shows that $\Delta$ almost surely contains only essential elements, thus with probability one $\Delta < A(\Gamma)$ and the Corollary is proved.
\end{proof}

\noindent {\bf Acknowledgments.}
I am grateful for significant input from Mikl\'{o}s Ab\'{e}rt, Peter Abramenko, Uri Bader, Tsachik Gelander, Eli Glasner, Dennis Osin, Robin Tucker-Drob, B\'{a}lint Vir\'{a}g and Kevin Wortman. I thank the referee for many helpful comments. I thank Shalvia for going over the English. This work was written while I was on sabbatical at the University of Utah. I am  very grateful to hospitality of the math department there and to the NSF grants that enabled this visit. The author acknowledges support from U.S. National Science Foundation grants DMS 1107452, 1107263, 1107367 ``RNMS: Geometric structures And Representation varieties" (the GEAR Network). The author was partially supported by the Israel Science Foundation grant ISF 441/11.
 
\bibliographystyle{alpha}
\bibliography{../tex_utils/yair}

\noindent {\sc Yair Glasner.} Department of Mathematics.
Ben-Gurion University of the Negev.
P.O.B. 653,
Be'er Sheva 84105,
Israel.
{\tt yairgl\@@math.bgu.ac.il}\bigskip

\noindent {\sc Tsachik Gelander} Department of Mathematics.
Weizmann Institute of Science.
Rehovot 76100,
Israel.
{\tt tsachik.gelander\@@weizmann.ac.il}\bigskip

\end{document}